\documentclass[10pt, a4paper]{article}
\usepackage {fullpage}						% set all margins to 1.5cm
\usepackage {ucs}
\usepackage [utf8x]{inputenc}
\usepackage {graphicx}
\usepackage [bookmarks=false, pdfstartview={FitH}, colorlinks=true, linkcolor=blue, citecolor=blue, urlcolor=blue]{hyperref} % create the .pdf links

\usepackage {amsfonts,amsmath,amsthm,amssymb}
\usepackage {latexsym, bbm}
\usepackage {graphics, graphicx}
\usepackage {fancyhdr}
\usepackage[table]{xcolor}
\usepackage {epstopdf} 						% to automatically transform .eps figs to .pdf

\usepackage {multicol}						% to locally have two columns
\usepackage {multirow}  					% merge rows in a tabular
\usepackage {tabularx}						% to centralize inside matrix columns
\newcolumntype{C}[1]{>{\centering\arraybackslash}p{#1}}
\usepackage {soul} 							% to break \underline to several lines using \ul
\usepackage {authblk} 						% more than one authors
\usepackage {datetime} 						% date and time in latex
\usepackage {subfigure}						% to have enumerated subfigures in one figure
\usepackage {parskip}						% no indentation and space between paragraphs
\usepackage {setspace}						% to set the vertical spacing
\usepackage[font=small,textfont=it,width=0.97\textwidth]{caption}		% to set the fonts for the captions, to use caption*
\usepackage {mathtools}						% to have dcases environment
\usepackage {array}							% to vertically align within a table
\usepackage{filemod}						% last modification date and time of the file

\usepackage{constants}

\newtheorem{theorem}{Theorem}[section]
\newtheorem{remark}[theorem]{Remark}
\newtheorem{lem}[theorem]{Lemma}

\newconstantfamily{h}{symbol=h}

\numberwithin{equation}{section}

\renewcommand{\(}{\left(}
\renewcommand{\)}{\right)}

\newcommand{\cs}{c^\text{S}}
\newcommand{\cd}{c^\text{D}}

\newcommand{\as}{a^\text{S}}
\newcommand{\ad}{a^\text{D}}

\newcommand{\memt}{\mu_\text{EMT}}

\def\softd{{\leavevmode\setbox1=\hbox{d}%
		\hbox to 1.05\wd1{d\kern-0.4ex{\char039}\hss}}}%cstocs

\newcommand{\aos}{a_{1*}}
\newcommand{\ats}{a_{2*}}	
\newcommand{\ados}{\ad_{1*}}
\newcommand{\asos}{\as_{1*}}
\newcommand{\vos}{v_{1*}}

\newcommand{\rdev}{\rho_\text{dev}}

\newcommand{\vfree}{\rdev}
\newcommand{\vfreeo}{\rho_{\text{dev,} 1}}
\newcommand{\vfreet}{\rho_{\text{dev,} 2}}

\newcommand{\adts}{\ad_{2*}}
\newcommand{\asts}{\as_{2*}}
\newcommand{\vts}{v_{2*}}
\newcommand{\coz}{{C^{1,0}(\bar Q_T)}}
\newcommand{\coo}{{C^{1,1}(\bar Q_T)}}
\newcommand{\hoeloh}[1]{{C^{1+ #1, (1 + #1)/2}(\bar Q_T)}}
\newcommand{\hoels}[2]{{C^{#1, #2}(\bar Q_T)}}
\newcommand{\hoelohz}[1]{{C^{1+ #1, (1 + #1)/2}(\bar Q_{T_0})}}
\newcommand{\hoelsz}[2]{{C^{#1, #2}(\bar Q_{T_0})}}
\newcommand{\sobto}{{W_p^{2,1}(Q_T)}}
\newcommand{\linf}{{L^\infty(Q_T)}}
\newcommand{\linfo}{{L^\infty(\Omega)}}

\newcommand{\lqn}{{L^q(\Omega)}}
\newcommand{\lpn}{{L^p(\Omega)}}

\newcommand{\ai}{a^I}
\newcommand{\aj}{a^J}

\newcommand{\adg}{\ad_\gamma}
\newcommand{\asg}{\as_\gamma}
\newcommand{\adgp}{(\adg)^p}
\newcommand{\asgp}{(\asg)^p}
\newcommand{\adgph}{(\adg)^{p/2}}
\newcommand{\asgph}{(\asg)^{p/2}}

\newcommand{\memto}{\mu_{\text{EMT},1}}
\newcommand{\memtt}{\mu_{\text{EMT},2}}

\newcommand {\startenv} {\vskip 0.05em
	\begin{tabular}{||l}\parbox[t]{0.95\linewidth}}
	\newcommand {\stopenv} {\end{tabular}\vskip 0.05em}

\title{Existence and uniqueness of global classical solutions to a two species cancer invasion haptotaxis model}

\author[1]{Jan Giesselmann\thanks{jgiessel@mathematik.uni-stuttgart.de}}
\author[2]{Niklas Kolbe\thanks{kolbe@uni-mainz.de}}
\author[2]{M\'aria Luk\'{a}\v{c}ov\'{a}-Medvi\softd ov\'{a}\thanks{lukacova@uni-mainz.de}}
\author[2,3]{Nikolaos Sfakianakis\thanks{sfakiana@math.uni-heidelberg.de}}
\affil[1]{Institute of Applied Analysis and Numerical Simulation, University of Stuttgart}
\affil[2]{Institute of Mathematics, Johannes Gutenberg-University Mainz}
\affil[3]{Institute of Applied Mathematics, University of Heidelberg}

\begin{document}
	\maketitle
%\comm{some journals to consider (alphabetic listing): 1) Communications in mathematical sciences (editor Perthame), 2) Discrete and Continuous Dynamical Systems B (editors Schmeiser, Perthame), 3) Journal of Differential Equations, 4) Nonlinearity
%}

%\comm{possible reviewers to suggest: Tao, Stevens, Surulescu, Stinner, Winkler}	
\begin{abstract}
	We consider a haptotaxis cancer invasion model that includes two families of cancer cells. Both families, migrate on the extracellular matrix and proliferate.
	Moreover the model describes an epithelial-to-mesenchymal-like transition between the two families, as well as a degradation and a self-reconstruction process of the extracellular matrix. We prove positivity and conditional global existence and uniqueness of the classical solutions of the problem for large initial data.
\end{abstract}

%%%%%%%%%%%%%%%%%%%%%%%%%%%%%%%%%%%%
\section{Introduction}
Cancer research is a multidisciplinary effort to understand the causes of cancer and to develop strategies for its diagnosis and treatment.
The involved disciplines include the medical science, biology, chemistry, physics, informatics, and mathematics. 
From a mathematical point of view, the study of cancer has been an active research field since the 1950s and addresses different biochemical processes relevant to the development of the disease,
see e.g. \cite{Preziosi.2003, Bellomo.2008, Vainstein.2012, Michor.2008, Roose.2007}. 

In particular, a large amount of the research focuses on the modelling of the \textit{invasion} of the \textit{Extracellular Matrix} (ECM);
the first step in \textit{cancer metastasis} and one of the \textit{hallmarks of cancer}, \cite{Hanahan.2000,  Perumpanani.1996, Chaplain.2005, Perthame.2014}. The invasion of the ECM, involves also a secondary family of cancer cells that is more resilient to cancer therapies. These cells are believed to possess \textit{stem cell-like} properties, such as self-renewal and differentiation, as well as the ability to metastasize, i.e. detach from the primary tumour, afflict secondary sites within the organism and engender new tumours \cite{Brabletz.2005,Johnston.2010}. These cells are termed \textit{Cancer Stem Cells} (CSCs) and originate from the more usual \textit{Differentiated Cancer Cells} (DCCs) via a cellular differentiation program that is related to another cellular differentiation program found also in normal tissue, the \textit{ Epithelial-Mesenchymal Transition} (EMT) \cite{Mani.2008, Gupta.2009, Reya.2001}.

\begin{figure}[t]
	\centering
	\includegraphics[width=0.4\linewidth]{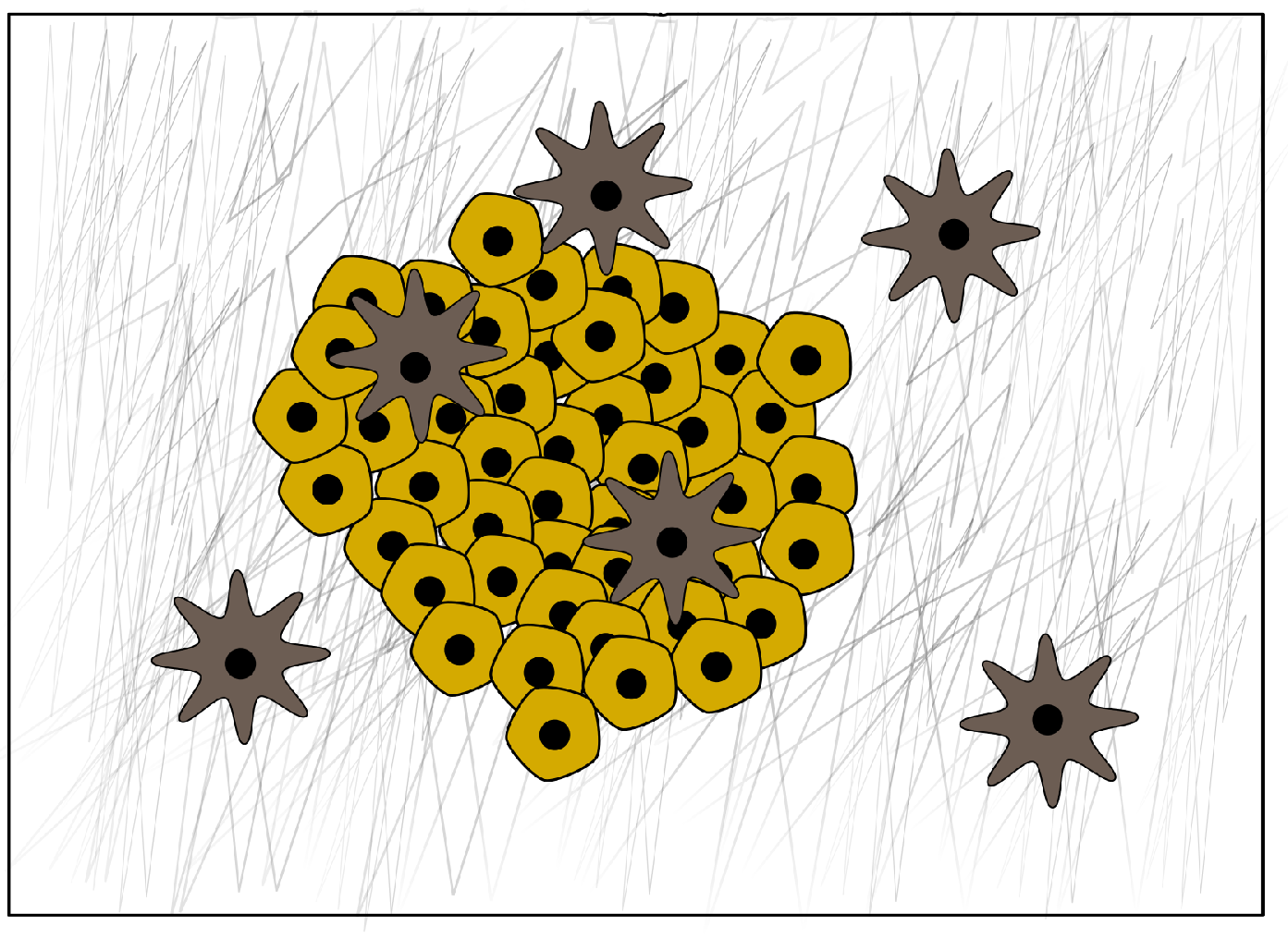}
	
	\begin{tabular}{C{4em}C{4em}C{4em}}
		\includegraphics[width=1.5em]{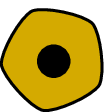}
			&\includegraphics[width=2em]{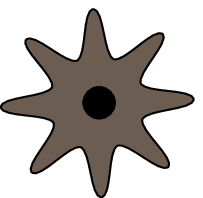}
			&\includegraphics[width=4em]{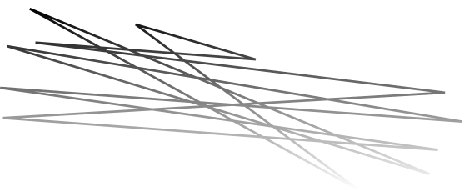}\\
		DCC&CSC&ECM
	\end{tabular}
	\caption{Graphical description on the model \eqref{eq:reduced_II}. The more aggressive CSCs escape the main body of the tumour and invade the ECM faster than the DCCs. At the same time, cancer secreted MMPs degrade the ECM.}
	\label{fig:modelgraphically}
\end{figure}

Both types of cancer cells invade the ECM and while doing so, affect its architecture, composition, and  functionality. One of the methods they use, is to secrete \textit{matrix metalloproteinases}  (MMPs),
i.e. enzymes that degrade the ECM and allow for the cancer cells to move through it more freely, \cite{Gross.1962,Egeblad.2002}.

During the EMT and the subsequent invasion of the ECM, \textit{chemotaxis}\footnote{cellular movement under the influence of one or more chemical stimuli},
and \textit{haptotaxis}\footnote{cellular movement along gradients of cellular adhesion sites or ECM bound chemoattractants}, play fundamental role \cite{Condeelis.2011,Rao.2003}.
These processes are typically modelled using Keller-Segel (KS) type systems,
i.e. macroscopic deterministic models that were initially developed to describe the chemotactic movement and aggregation of \textit{Dictyostelium discoideum} bacteria.
These models were introduced in \cite{Patlak.1953,KS.1970} and were later (re-)derived using a many particle system approach in \cite{Stevens.2000}.
They are known to potentially (according to the spatial dimension and the initial mass) blow-up in finite time and their analysis has been a field of intensive research, e.g. \cite{Blanchet.2006, Schmeiser.2009}.

In a similar spirit, KS-like models have been used to model cancer invasion while taking into account chemotaxis, haptotaxis, and other processes important in development of cancer, see e.g. \cite{Anderson.2000,Szymanska.2009}. Although these models are simplifications of the biochemical reality of the tumour, their solutions display complex dynamics  and their mathematical analysis is challenging. We refer indicatively to some relevant results on the analysis of these models. It is by far not an exhaustive list of the topic, rather an insight to analytical approaches for similar models. 

In \cite{Czochra.2010} a single family of cancer cells is considered. The model is haptotaxis with cell proliferation, matrix degradation by the MMPs, without matrix remodelling. In this work global existence of weak solutions is proven. In addition, the solutions are shown to be uniformly bounded using the method of ``bounded invariant rectangles'', which can be applied once the model is reformulated in divergence form using a particular change of variables. 

In \cite{tao2011global} the author considers a haptotaxis model with one type of cancer cells, which accounts for self-remodelling of the ECM, and ECM degradation by MMPs.
With respect to the MMPs, the model is parabolic. The decoupling between the PDE governing the cancer cells, and the ODE describing the ECM, is facilitated by a particular non-linear change of variables. The global existence of classical solutions follows by a series of delicate a-priori estimates and corresponding limiting processes.

In \cite{Winkler.2014} a single family of cancer cells is considered that responds in chemotactic-haptotactic way to its environment. The ECM is degraded by the MMPs and is self-remodelled. The diffusion of the MMPs is assumed to be very fast and the resulting equation is elliptic. Global existence of classical solutions follows after a-priori estimates, that are established using energy-type arguments.

In \cite{Stinner.2015} two species of cancer cells are considered using a motility-proliferation dichotomy hypothesis on the cancer cells. Further assumptions include the matrix degradation and (self-)remodelling, as well as a type of radiation therapy. The authors prove global existence of weak solutions via an appropriately chosen ``approximate'' problem and entropy-type estimates. 

For further results on the analysis of similar models we refer to the works \cite{Corrias.2004, horstmann2005boundedness, Walker.2007, Tao.2009, Hillen.2013}.

% \comm{nks: Should we include references on the analysis of classical (bacterial) Keller-Segel models. This would include a series of works by perthame, schmeiser, and many others...;
% Before discussing the work of Czochra I would suggest we mention that KS equations are well-known for blow-op phenomena for certain initial data, and give a few classical references.
% Then, we say that many extended(?) models admit global strong solutions. Is there some review paper on this which we could cite? nks: I made some additions.}

In our paper the cancer invasion model features DCCs, with their density denoted by $\cd$, CSCs, denoted as $\cs$, and the EMT transition between them.
We consider the model in two space dimensions and  assume that both families of cancer cells perform a haptotaxis biased random motion modelled by the combination of diffusion and advection terms.
We assume moreover that they proliferate with a rate that is influenced by the local density of the total biomass.
The ECM $v$ is assumed to be degraded by the MMPs $m$ which in turn are produced by the cancer cells. They diffuse freely in the environment and degrade with a constant rate.

The model proposed in \cite{Sfakianakis.2017, Hellmann.2016} reads as follows:

\begin{equation}\label{eq:reduced_II}
\left\{
	\begin{aligned}
		\cd_t
			&= \Delta \cd -\chi_D\nabla \cdot \(\cd \nabla v\right) - \memt \, \cd
					+ \mu_D \, \cd (1-\cs-\cd-v)\phantom{^+},\\
		\cs_t
			&= \Delta \cs -\chi_S\nabla \cdot \(\cs \nabla v\right) 
%					\phantom{- \mtra \,  \cs} 
					+ \memt \, \cd + \mu_S \, \cs (1-\cs-\cd-v)\phantom{^+},\\
%			\phantom{\frac{\partial \cf}{\partial t}}
%				&\phantom{= D_F \Delta \cf +\chi_F\nabla \cdot \(\cf \nabla v\right) 
%					\phantom{+ \mtra\, \cs} 
%					+ \mu_F \, \cf (1-\cs-\cd\phantom{-\cf}-v)\phantom{^+}  
%					-\beta_F \,\cf}\\
		v_t
			&=-m v + \mu_v\, v (1-\cs-\cd-v),\\
		m_t
			&=\Delta m + \cs + \cd - m,
	\end{aligned}
	\right.
\end{equation}
with (fixed) coefficients $\chi_D, \chi_S, \mu_S, \mu_D , \mu_v >0$ and an EMT rate function $\memt$ whose properties will be specified below.

The system \eqref{eq:reduced_II} is complemented with the no-flux boundary conditions
\begin{equation}\label{eq:bcs_c}	
	\partial_\nu \cd - \chi_D\cd \partial_\nu v = \partial_\nu \cs - \chi_D\cs \partial_\nu v  = \partial_\nu m = 0\quad \text{in }\partial \Omega \times (0,T)
\end{equation}
and the initial data
\begin{equation}\label{eq:ics_c}
	\cd(\cdot, 0) = \cd_0,~\cs(\cdot, 0) = \cs_0,~v(\cdot, 0) = v_0,~m(\cdot, 0) = m_0\quad \text{on }\Omega,
\end{equation}
for which we assume that
\begin{equation}\label{eq:iniData_c}
	\cd_0, \cs_0, m_0 \geq 0,\quad 0\leq v_0 \leq 1,\quad \cd_0, \cs_0, m_0, v_0 \in C^{2+l}({\bar \Omega}),
\end{equation}
for a given $0 < l < 1$. The domain $\Omega \subset \mathbb R^2$  is bounded with smooth boundary $\partial \Omega$ that satisfies
\begin{equation}\label{eq:smoothDom}
\partial \Omega \in C^{2+l}.
\end{equation}

The model \eqref{eq:reduced_II} has been scaled with respect to reference values of the primary variables
and the coefficients of diffusion as well of the evolution of the MMPs  have been reduced to 1 since they do not participate in the final (conditional) global existence result.
For the complete coefficient/parameter set we refer to \cite{Sfakianakis.2017}.

%Moreover, %we consider the space dimension $d=2$, although the local existence results are valid also for higher dimensions
%\comm{jn:  I have removed $d=2$ here. I have included it in the intro instead. It is also included implicitly in the line prior to (1.5). I would suggest we replace $d$ by $2$ everwhere.}
We moreover assume that the parameters of the problem satisfy
\begin{equation}\label{eq:mud_bound}
	\mu_D\geq \chi_D \mu_v, \quad \mu_S\geq \chi_S \mu_v.
\end{equation}
This condition is crucial for the analysis presented in this paper. Similarly to the open problem posed at the end 
of \cite{tao2011global} it is not clear whether solutions to \eqref{eq:reduced_II} may blow up in case \eqref{eq:mud_bound}
does not hold.
%\comm{nks: could we alleviate or relax the second one?}

We assume that the EMT rate $\memt$ is a function $\memt: \mathbb{R}^4 \rightarrow \mathbb{R}$, that is Lipschitz continuous, has Lipschitz continuous first derivatives, and satisfies moreover for $\mu_M> 0$,
\begin{subequations}

\begin{equation} \label{eq:memt_bound}
	0 \leq \memt \leq \mu_M.
\end{equation}
Due to the continuity, we get for $\memt$ that, 
%	is a function of $(\cd, \cs, v, m)$
%	that satisfies
%	\begin{equation} \label{eq:memt_diff}
%		|\nabla_{(\cd, \cs, v, m)} \memt| \leq \C
%	\end{equation}
%			Due to \eqref{eq:memt_diff}, $\memt$ 
	%		is Lipschitz continuous in the sense that
\begin{align}
	\|\memt(\cd_1, \cs_1, v_1, m_1) - \memt(\cd_2, \cs_2, v_2, m_2) \|_\coz \leq & L ( \|\cd_1 - \cd_2\|_\coz + \|\cs_1 - \cs_2\|_\coz  \nonumber \\  &+ \|v_1 - v_2\|_\coz + \|m_1 - m_2\|_\coz ),  \label{eq:memt_lipschitz}
\end{align}
for an $L$ with
\begin{equation}
	L = \C \|(\cd_1, \cs_1, v_1, m_1) \|_\coz. \label{eq:memt_lipschitz2}
\end{equation}
 
\end{subequations}
Here $\bar Q_T$ is the closure of the cylinder
\begin{equation}\label{eq:cylinder}
	Q_T := \Omega \times (0,T).
\end{equation}

Let us note that throughout this work we will call solutions of \eqref{eq:reduced_II} {\it strong solutions} provided they are regular enough that all derivatives appearing in \eqref{eq:reduced_II} are weak and the solution belongs to the corresponding Sobolev space, e.g. $W^{2,1}_p(Q_T)$.
We refer to solutions of \eqref{eq:reduced_II} as {\it classical solutions} provided their regularity is such that all terms in \eqref{eq:reduced_II}  are point wise well-defined.
The main result in this work %constitutes of proving 
is the proof of existence and uniqueness of global classical solutions to the problem \eqref{eq:reduced_II}.

\begin{theorem}[Global existence]\label{thm:global.existence}
	Let $d=2$ and \eqref{eq:mud_bound} hold.
	Then for any $T>0$ and $0<l<1$ there exists a unique classical solution 
	\begin{equation*}
	(\cd, \cs, v, m) \in (\hoels{2+l}{1+l/2}) ^4,
	\end{equation*}
	 of the system \eqref{eq:reduced_II}--\eqref{eq:smoothDom} with $\cd, \cs, m \geq 0$ and $0 \leq v \leq 1$.
\end{theorem}

The proof of Theorem \ref{thm:global.existence} is based on a local existence result for strong solutions, Theorem \ref{loc_existence}, a proof that the strong solutions are indeed classical solutions, Theorem \ref{thm:reg_local},
and a series of a-priori estimates, inspired by \cite{tao2011global},
that enable us to extend the local solutions for  large times. We note that the raise of the regularity, which takes place in Lemma \ref{lem:final_a_priori}, could not be achieved by means of energy-type techniques as in \cite{tao2011global}.
We instead base our argumentation on parabolic $L^p$ theory and Sobolev embeddings, using an approach that resembles the strategy employed in \cite{Winkler.2014}.

Comparing this work with \cite{Czochra.2010,tao2011global,Winkler.2014} we note that the model \eqref{eq:reduced_II} features two types of cancer cells.
We treat their corresponding equations separately due to the different motility parameters of the two families, but their non-linear coupling by the EMT necessitates particular treatment.
In comparison to \cite{Stinner.2015} the model we consider in this work assumes that both families of  cancer cells migrate and proliferate and that the EMT takes place only in one direction. Thus, we do not consider \textit{mesenchymal-epithelial transition}. Moreover, we allow for a wide variety of EMT coefficient (functions) that are bounded and Lipschitz continuous \eqref{eq:memt_bound}.

%\env{
%\comm{nks: I believe that we still need a sentence as to why this work should be published. if oyu have something good, please go ahead, otherwise we leave it as is.}
%\\
%Overall, the major components that the system \eqref{eq:reduced_II} includes e.g. the two types of migrating and proliferating cancer cells and the non-constant EMT transition between them, in combination with the global existence and positivity results that we prove, as well as the combination of techniques that we use, make this paper \remm{worth publishing.}
%}
 
The rest of this paper is structured as follows: in Section \ref{sec:lear} we perform a change of variables and prove local existence of strong solutions by a fixed point argument.
In addition, we show that these strong solutions are classical solutions.
Section \ref{sec:apriori} is devoted to a series in of a-priori estimates which continues in Section \ref{sec:apriori-v}.
These estimates allow us to extend the local solutions to global solutions in Section \ref{sec:global}. We conclude with two appendices. Appendix A gathers some facts from parabolic theory and Appendix B contains the proof of a technical lemma.

%---------------------------------------------------------------------------------------
\section{Local existence of classical solutions regularity}\label{sec:lear}
In this section we show local in time existence of classical solutions.
To this end we reformulate \eqref{eq:reduced_II} using a change of variables.
\subsection{Change of variables}
	Following \cite{tao2011global, Winkler.2014} we perform the change  of variables

	\[ \begin{dcases} \ad &=\cd e^{-\chi_D v} \\  \as & =\cs e^{-\chi_C v} \end{dcases}.
		%\Longleftrightarrow		
		%	\begin{dcases} \cd= e^{\chi_D v} \ad  \\ \cs= e^{\chi_S v} \as  \end{dcases} 
		\]
	Consequently, the system \eqref{eq:reduced_II} recasts as

	\begin{equation}\label{eq:system_a}
	\left\{	\begin{aligned}
		\ad_t &= e^{-\chi_D v} \nabla \cdot \(e^{\chi_D v} \nabla\ad\) + \chi_D \ad vm - \memt \, \ad 
					+ (\mu_D-\chi_D \mu_v\,v)\ad \rdev\\
		\as_t &= e^{-\chi_S v} \nabla \cdot \(e^{\chi_S v} \nabla\as\)  + \chi_S \as vm + \memt \, \ad 
					+ (\mu_S-\chi_S \mu_v\,v)\as \rdev\\
		m_t &= \Delta m +  e^{\chi_S v} \as  +  e^{\chi_D v} \ad  - m,\\
		v_t &= -m v + \mu_v\, v \rdev			
	\end{aligned},\right.
	\end{equation}
	where %we have set  
	\begin{equation}\label{eq:pdev}
		\rdev= 1- e^{\chi_S v}\as- e^{\chi_D v} \ad -v
	\end{equation}
	describes the deviation of the \textit{total density} from the equilibrium value 1. %\comm{nkls: maybe say something about the scaling here ore somewhere else.} \comm{nks: something like?: This equilibrium value results from the rescaling  of the problem with respect to, what is called in the literature, as carrying capacity of the organism.}
		
	The system is closed with initial and boundary conditions resulting from \eqref{eq:bcs_c} and \eqref{eq:ics_c}
	\begin{equation}\label{eq:bcs_a}\left\{	\begin{aligned}
	\partial_\nu \ad = \partial_\nu \as  = \partial_\nu m = 0\quad \text{in }\partial \Omega \times (0,T) \\
	\ad(\cdot, 0) = \ad_0,~\as(\cdot, 0) = \as_0,~v(\cdot, 0) = v_0,~m(\cdot, 0) = m_0\quad \text{on }\Omega,\end{aligned} \right.
	\end{equation}
	Analogously, \eqref{eq:iniData_c} implies
	\begin{equation}\label{eq:iniData}
	\ad_0, \as_0, m_0 \geq 0,\quad 0\leq v_0 \leq 1,\quad \ad_0, \as_0, m_0, v_0 \in C^{2+l}({\bar \Omega}).
	\end{equation}

	For the rest of this work we will use the following notation: 
	\begin{equation}
	\left\{
	 \begin{aligned}
	  W^{2,1}_p(Q_{T}) &= \{ u: Q_T \rightarrow \mathbb{R} | u, \nabla u, \nabla^2 u , \partial_t u \in L^p(Q_T)\},\\
	  W^{2}_p(\Omega) &= \{ u: Q_T \rightarrow \mathbb{R} | u, \nabla u, \nabla^2 u , \partial_t u \in L^p(Q_T)\},\\
	  C^{1,1}(\bar Q_{T}) &= \{ u: Q_T \rightarrow \mathbb{R} | u, \nabla u, \partial_t u \in C^0(\bar Q_T)\},\\
	  C^{1,0}(\bar Q_{T})&= \{ u: Q_T \rightarrow \mathbb{R} | u, \nabla u \in C^0(\bar Q_T)\}.
	 \end{aligned}
	 \right.
	\end{equation}
%	\comm{nkls: Maybe add definition for Hölder spaces? nks: I don't mind}
%---------------------------------------------------------------------------------------
%---------------------------------------------------------------------------------------
	\subsection{Local existence}
	In this section we establish existence and uniqueness of local (in time) classical solutions of \eqref{eq:system_a}. We begin by showing existence and uniqueness of local (in time) strong solutions.
	\begin{theorem}[Local existence and uniqueness]\label{loc_existence} Let  \eqref{eq:iniData} and  \eqref{eq:smoothDom} be satisfied. Then there exists
		a unique strong solution $(\ad, \as, v, m) \in W^{2,1}_p(Q_{T_0}) \times W^{2,1}_p(Q_{T_0}) \times C^{1,1}(\bar Q_{T_0})\times W^{2,1}_p(Q_{T_0})$
		(for any $p>5$) of system \eqref{eq:system_a}, \eqref{eq:bcs_a} for a final time $T_0 >0$ depending on
		$$ M = 3 \|\ad_0 \|_{C^2} + 3 \|\as_0 \|_{C^2} + 9 \|v_0 \|_{C^1} + \|m_0 \|_{C^2}+ 3.$$ Moreover,
		$$ \ad, \as, m \geq 0,\quad 0\leq v \leq 1.$$
	\begin{proof} 
		
	We will prove the local existence by Banach's fixed point theorem
	\paragraph{Spaces.} Let $X$ be the Banach space of functions $(\ad, \as, v)$ with finite norm
	$$ \|(\ad, \as, v)\|_X =  \| \ad\|_{C^{1.0}(\bar Q_T)} + \| \as\|_{C^{1.0}(\bar Q_T)} + \| v\|_{C^{1.0}(\bar Q_T)},\quad 0<T<1$$
	and 
	$$ X_M := \left\{ (\ad, \as, v)\in (\coz)^3: \ad, \as, v \text{ satisfy \eqref{eq:bcs_a}, and } \|(\ad, \as, v)\|_X \leq M \right\}.$$
	
	\paragraph{Fixed point.} For any $(\ad, \as, v)\in X_M$ we define $(\ad_*, \as_*, v_*) = F(\ad, \as, v)$ given such that 
	\begin{subequations}
	\begin{align}
	\label{eq:m_linpar1} m_t - \Delta m + m = \ad e^{\chi_D v}+ \as e^{\chi_S v}\quad \text{ in }Q_T,\\
	\label{eq:m_linpar2} \partial_\nu m = 0 \text{ in }\partial \Omega \times (0,T),\quad m(\cdot, 0) = m_0\text{ in } \Omega ,\\
	\label{eq:v_lin1}{v_*}_t = - m v_* + \mu_v v_* \rdev\quad \text{ in }Q_T, \\
	\label{eq:v_lin2}v_*(\cdot, 0) = v_0, \\
	\label{eq:ad_lin1}{\ad_*}_t - \Delta \ad_* - \chi_D \nabla v_* \cdot \nabla \ad_* + [\mu_\text{EMT}-(\mu_D - \chi_D \mu_v v) \rdev] \ad_* = \chi_D \ad v m,\\
	\label{eq:ad_lin2}\partial_\nu \ad_* = 0\text{ in }\partial \Omega \times (0,T),\quad \ad_*(\cdot, 0) = \ad_0\text{ in } \Omega, \\
	\label{eq:as_lin1}	{\as_*}_t - \Delta \as_* - \chi_S \nabla v_* \cdot \nabla \as_* - (\mu_S - \chi_S \mu_v v) \rdev \as_* = \chi_S \as v m + \mu_\text{EMT} \ad_*,\\
	\label{eq:as_lin2}	\partial_\nu \as_* = 0\text{ in }\partial \Omega \times (0,T),\quad \as_*(\cdot, 0) = \ad_0\text{ in } \Omega,
	\end{align}
	\end{subequations}
	where  $\rdev$ is given by \eqref{eq:pdev}. %= 1-\as e^{\chi_S v}-\ad e^{\chi_D v}-v$.
	 For the proof we fix some (arbitrary) $p>5$ and set $\lambda = 1- \frac{p}{5}$.
	\paragraph{$F$ is well defined and $F(X_M) \subset X_M$.} 
	We start with the component $m$ and consider the equations \eqref{eq:m_linpar1}-\eqref{eq:m_linpar2}.
	Since $0<T<1$ and $(\ad, \as, v)\in X_M$ this linear parabolic problem has a unique solution by Theorem \ref{thm:parabolic_lp}:
	\begin{equation}
		\|m\|_\sobto \leq \C(M) \label{eq:m_comp}.
	\end{equation}
	Here we can apply the Sobolev embedding Theorem \ref{thm:sob_embed_hoelder} and get
	\begin{equation}
		\|m \|_{C^{1,0}(\bar Q_T)} \leq \C(M). \label{eq:m_in_c01}
	\end{equation}
	Moreover, the parabolic comparison principle yields
	\begin{equation}
		m \geq 0. \label{eq:m_positive}
	\end{equation}
	The initial value problem \eqref{eq:v_lin1}, \eqref{eq:v_lin2} can be written as
	\begin{equation}
	v_{*t} = \Cl[h]{oh1} v_*, \quad v_*(\cdot, 0) = v_0, \label{eq:vODEsimple}
	\end{equation}
	where 
	\begin{equation}
		\| \Cr{oh1}\|_\coz = \| -m + \mu_v \rdev\|_\coz \leq \Cl{oc1}(M) \label{eq:vODEestder}
	\end{equation}
	due to \eqref{eq:m_in_c01} and $(\as, \ad, v) \in X_M$. The ODE system has the solution
	\begin{equation}
		v_* = v_0(x) \exp\(\int_{0}^{t} \Cr{oh1}(x,s) ds\) \geq 0 \label{eq:vstar_solution}
	\end{equation}
	with gradient
	 	\begin{equation}
	 	\nabla v_* = \nabla v_0(x) \exp\(\int_{0}^{t} \Cr{oh1}(x,s) ds\) + v_0(x) \exp\(\int_{0}^{t} \Cr{oh1}(x,s) ds\)
	 	\int_{0}^{t} \nabla \Cr{oh1}(x,s) ds.
	 	\end{equation}
	 	For $T\leq \frac{1}{2 \Cr{oc1}(M)}< \log(2)/\Cr{oc1}(M)$ we get
	 	\begin{align}
	 		\| v_* \|_{C(\bar Q_T)} &\leq \| v_0\|_{C(\bar \Omega)} e^{\Cr{oc1}(M)T} \leq 2  \| v_0\|_{C(\bar \Omega)} \\
	 	\nonumber		\| \nabla v_* \|_{C(\bar Q_T)} &\leq \| \nabla v_0(x)\|_{C(\bar \Omega)} \exp(\Cr{oc1}(M)T) + \|v_0(x)\|_{C(\bar \Omega)} \exp(\Cr{oc1}(M)T) T\Cr{oc1}(M)  \\
	 	&\leq 2	\| \nabla v_{0} \|_{C(\bar \Omega)} + \| v_{0} \|_{C(\bar \Omega)}	 	
	 	\end{align}
	 	and thus
	 	\begin{equation}
	 		\| v_*\|_{C^{1,0}(\bar Q_T)} = \| v_* \|_{C(\bar Q_T)} + \| \nabla v_* \|_{C(\bar Q_T)}
	 		\leq 3  \| v_{0} \|_{C(\bar \Omega)}	+ 2 \| \nabla v_{0} \|_{C(\bar \Omega)} \leq 3  \| v_{0} \|_{C^1(\bar \Omega)} \leq M/3. \label{eq:vbound}
	 	\end{equation}

	Next, we deal with the parabolic problem \eqref{eq:ad_lin1}, \eqref{eq:ad_lin2} that can be written as
	\begin{equation}
		a_{*t}-\Delta a_* - \chi \nabla v_* \cdot \nabla a_* - \Cl[h]{h2}a_* =\Cl[h]{h3}  \label{eq:parabol_lin_a}
	\end{equation}
	with boundary and initial conditions given by \eqref{eq:ad_lin2} where $a_* = \ad_*$, $\chi=\chi_D$. We have
	\begin{equation} \label{eq:apde_coeffs}
		\| \nabla v_* \|_{L^\infty(Q_T)} \leq M, \quad \| \Cr{h2} \|_{L^\infty(Q_T)} \leq \C (M), \quad \| \Cr{h3} \|_{L^\infty(Q_T)} \leq \C (M),
	\end{equation}
	because of $(\ad, \as, v)\in X_M$, \eqref{eq:m_in_c01}, \eqref{eq:memt_bound}. Applying the maximal parabolic regularity result (Theorem \ref{thm:parabolic_lp}), there is a unique solution $a_*$ that satisfies
	\begin{equation}
		\|a_* \|_{W^{2,1}_p(\bar Q_T)} \leq \C(M)\quad \forall p>1. \label{eq:a_in_sob}
	\end{equation}  
	Further the Sobolev embedding \ref{thm:sob_embed_hoelder}: $W_p^{2,1}(\bar Q_T) \hookrightarrow C^{1+\lambda, (1+\lambda)/2}(\bar Q_T)$ gives us
	\begin{equation}
		\|a_* \|_{C^{1+\lambda, (1+\lambda)/2}(\bar Q_T)} \leq \Cl{c5}(M).
	\end{equation}
	If $T\leq \Cr{c5}(M)^\frac{-2}{1 + \lambda}$ we get
	\begin{align}
	\nonumber	\|a_* \|_{C^{1,0}} & = \|a_* \|_{C^0(\bar Q_T)} + \|\nabla a_* \|_{C^0(\bar Q_T)} \\
	\nonumber	&\leq \|a_* - a_0 \|_{C^0(\bar Q_T)} + \| a_0 \|_{C^0(\bar \Omega)} + \|\nabla a_* - \nabla a_0\|_{C^0(\bar Q_T)}
			+ \| \nabla a_0 \|_{C^0(\bar \Omega)}	 \\
	\nonumber	&\leq T^{(1+\lambda)/2} \|a_* \|_{C^{1, (1+\lambda)/2}(\bar Q_T)} + \| a_0 \|_{C^1(\bar \Omega)} \\
	\nonumber	&\leq T^{(1+\lambda)/2} \Cr{c5}(M) + \| a_0 \|_{C^1(\bar \Omega)} \\
	\nonumber & \leq 1 +  \| a_0 \|_{C^1(\bar \Omega)}\\
			&\leq M/3. \label{eq:abound}
	\end{align}
	Moreover, 
	\begin{equation}
		a_* \geq 0 \label{eq:a_positive}
	\end{equation} by the parabolic comparison principle since the right hand side of \eqref{eq:ad_lin1} is non negative.
	Since we have shown that $\ad_* \in X_M $, the assertion \eqref{eq:apde_coeffs} is true also for $a= \as$ in the problem \eqref{eq:parabol_lin_a}. Hence \eqref{eq:abound}, \eqref{eq:a_positive} for $a_*= \as_*$ follow by the same arguments.
	
	\paragraph{$F$ is a contraction.}	
	We take $(\ad_1, \as_1, v_1), (\ad_2, \as_2, v_2) \in X_M$ and consider $ (\ados, \asos, \vos) = F(\ad_1, \as_1, v_1),$ $(\adts, \asts, \vts) = F(\ad_2, \as_2, v_2)$. As shown before one can find 
	$$m_1, m_2, \quad \| m_1\|_{C^{1,0}(\bar Q_T)} , \| m_2\|_{C^{1,0}(\bar Q_T)} \leq \C(M) $$
	that satisfy \eqref{eq:m_linpar1}, \eqref{eq:m_linpar2} for $(\ad, \as, m, v) = (\ad_1, \as_1, m_1, v_1), (\ad_2, \as_2, m_2, v_2)$. Further we have 
	\begin{align}
 (m_1 - m_2)_t - \Delta (m_1 - m_2) + (m_1 - m_2) = \ad_1 e^{\chi_D v_1}+ \as_1 e^{\chi_S v_1}
 - \ad_2 e^{\chi_D v_2}- \as_2 e^{\chi_S v_2}\quad \text{ in }Q_T,\label{eq:mdiff1}\\
\partial_\nu (m_1 - m_2) = 0 \text{ in }\partial \Omega \times (0,T),\quad (m_1-m_2)(\cdot, 0) = 0 \text{ in } \Omega, \label{eq:mdiff2}
	\end{align}
	where
	\begin{align}
	\| \ad_1 e^{\chi_D v_1}+ \as_1 e^{\chi_S v_1}
	&- \ad_2 e^{\chi_D v_2}- \as_2 e^{\chi_S v_2}\|_\linf \nonumber \\ &\leq \|  e^{\chi_D v_1} (\ad_1 - \ad_2)\|_\linf + \|  e^{\chi_S v_1} (\as_1 - \as_2)\|_\linf \nonumber\\
	&+ \| (e^{\chi_D v_1} - e^{\chi_D v_2}) (\ad_2) \|_\linf + \| (e^{\chi_S v_1} - e^{\chi_S v_2}) (\as_2) \|_\linf \nonumber\\
	&\leq \C(M) (\|\ad_1 - \ad_2 \|_\linf + \|\as_1 - \as_2 \|_\linf + \| v_1 - v_2\|_\linf).
	\end{align}
	Hence by Theorem \ref{thm:parabolic_lp} there is a solution to \eqref{eq:mdiff1},\eqref{eq:mdiff2} satisfying
	$$\| m_1 - m_2\|_\sobto \leq \C(M) (\|\ad_1 - \ad_2 \|_\linf + \|\as_1 - \as_2 \|_\linf + \| v_1 - v_2\|_\linf)$$
	for all $p>1$. The Sobolev embedding \ref{thm:sob_embed_hoelder} once again yields
	\begin{equation}\label{eq:mdiffbound}
	\| m_1 - m_2\|_\coz \leq \C(M) (\|\ad_1 - \ad_2 \|_\linf + \|\as_1 - \as_2 \|_\linf + \| v_1 - v_2\|_\linf).
	\end{equation}
	
	We get from \eqref{eq:v_lin1}, \eqref{eq:v_lin2} that 
	\begin{equation}\label{eq:vdiffode}
		(\vos - \vts )_t = \Cl[h]{h4} (\vos - \vts) + \Cl[h]{h5} , \quad (\vos - \vts)(\cdot, 0) = 0,
	\end{equation}
	where
	$$ \Cr{h4} = -m_1 + \mu_v \rho_{\text{dev,} 1}, \quad \Cr{h5} = (m_2-m_1)\vts + \mu_v \vts (\rho_{\text{dev,} 2} - \rho_{\text{dev,} 1}).$$
	There we have used the notation
	$$\vfreeo = 1- e^{\chi_S v_1}\as_1- e^{\chi_D v_1} \ad_1 -v_1, \quad \vfreet = 1- e^{\chi_S v_2}\as_2- e^{\chi_D v_2} \ad_2 -v_2. $$
	Since $(\ad_i, \as_i, v_i)\in X_M, ~i=1,2$ and due to \eqref{eq:mdiffbound}, we get
	\begin{align}
		\| \Cr{h4}\|_\coz &\leq \C(M), \label{eq:h4bound}\\
		\| \Cr{h5}\|_\coz &\leq \C(M) (\|\ad_1 - \ad_2 \|_\coz + \|\as_1 - \as_2 \|_\coz + \| v_1 - v_2\|_\coz). \label{eq:h5bound}
	\end{align}
	The solution of the ODE \eqref{eq:vdiffode} is given by
	\begin{equation}
			\vos - \vts = \int_{0}^{t} \exp\(\int_{0}^{t} \Cr{h4}(x,s) ds\) h_5(x, \tau) d\tau,
	\end{equation}
	and thus
		\begin{equation}
		\nabla(\vos - \vts) = \int_{0}^{t} \exp\(\int_{0}^{t} \Cr{h4}(x,s) ds\) \nabla_x \Cr{h5}(x, \tau) d\tau
		+ \int_{0}^{t} \exp\(\int_{0}^{t} \Cr{h4}(x,s) ds\) \Cr{h5}(x, \tau) \int_{0}^{t} \nabla_x \Cr{h4}(x,s) ds\, d\tau.
		\end{equation} Finally we obtain by using $0<T<1$ and the bounds \eqref{eq:h4bound}, \eqref{eq:h5bound} that
		\begin{align}
		\nonumber	\| \vos - \vts\|_\coz &\leq T \C(M) \| \Cr{h5}\|_\coz \\ &\leq T \Cl{c12} (M)(\|\ad_1 - \ad_2 \|_\coz + \|\as_1 - \as_2 \|_\coz + \| v_1 - v_2\|_\coz). \label{eq:vdiffbound}
		\end{align}
		
		Next, we derive the parabolic problem for $a \in \{\ad, \as\}$ with  coefficients $( \chi, \Cl[h]{h6}, \Cl[h]{h7})\in\{(\chi_D,\Cr{h6}^D, \Cr{h7}^D), (\chi_S,\Cr{h6}^S, \Cr{h7}^S)\}$
		by \eqref{eq:ad_lin1}--\eqref{eq:as_lin2}
		\begin{align}
			(\aos - \ats)_t - \Delta (\aos - \ats) - \chi \nabla \vos \cdot \nabla(\aos - \ats) + \Cr{h6} (\aos - \ats) =  \Cr{h7} \text{ in }Q_T, \label{eq:adiff1}\\
			\partial_\nu (m_1 - m_2) = 0 \text{ in }\partial \Omega,\quad(\aos - \ats)(\cdot, 0) = 0 \text{ on } \Omega, \label{eq:adiff2}
		\end{align}
		where
		\begin{align*}
			\Cr{h6}^D &= \memto - (\mu_D - \chi_D \mu_v v_1) \vfreeo, \quad \Cr{h6}^S = - (\mu_S - \chi_S \mu_v v_1)\vfreeo, \\
			\Cr{h7}^D &= \chi_D (a_1^D m_1 v_1 - a_2^D m_2 v_2) + \chi_D \nabla (\vos - \vts) \cdot \nabla \adts \\ &+
			\adts [(\mu_D - \chi_D \mu_v v_1) \vfreeo - (\mu_D - \chi_D \mu_v v_2) \vfreet- (\memto - \memtt )], \\
			\Cr{h7}^S &= \chi_S (a_1^S m_1 v_1 - a_2^S m_2 v_2) + \chi_S \nabla (\vos - \vts) \cdot \nabla \asts \\ &+
			\asts [(\mu_S - \chi_S \mu_v v_1) \vfreeo - (\mu_S - \chi_S \mu_v v_2) \vfreet] - (\memto \, \ados- \memtt \, \adts ).
		\end{align*}
		We have used the notation
		$$\memto = \memt(\cd_1, \cs_1, v_1, m_1), \quad \memtt = \memt(\cd_2, \cs_2, v_2, m_2). $$
		Due to $(\ad_i, \as_i, v_i)\in X_M$, \eqref{eq:vbound}, \eqref{eq:abound}, \eqref{eq:mdiffbound}, \eqref{eq:vdiffbound}, \eqref{eq:memt_lipschitz}, \eqref{eq:memt_lipschitz2} we can estimate
		\begin{align}
			\| \chi_D \nabla \vos \|_\linf , \| \chi_S \nabla \vos \|_\linf &\leq \C(M) \\
			\| h_6^D\|_\linf,  	\| h_6^S\|_\linf &\leq \C(M), \\
		\nonumber	\| h_7^D\|_\linf,  	\| h_7^S\|_\linf &\leq \C(M)(\|\ad_1 - \ad_2 \|_\coz &&+ \|\as_1 - \as_2 \|_\coz \\ & &&+ \| v_1 - v_2\|_\coz).
		\end{align}

		Since $0<T<1$ a solution of \eqref{eq:adiff1}, \eqref{eq:adiff2} exists by Theorem \ref{thm:parabolic_lp} with
		\begin{align*}
			\|\aos - \ats \|_\sobto &\leq \C(M) \|h_7\|_{L^p(Q_T)} \\
			& \leq \C(M)(\|\ad_1 - \ad_2 \|_\coz + \|\as_1 - \as_2 \|_\coz \\ & + \| v_1 - v_2\|_\coz),
		\end{align*} 
		hence the bound can be extended using the Sobolev embedding \ref{thm:sob_embed_hoelder} and we get
	\begin{equation}
		\|\aos - \ats \|_\hoeloh{\lambda} \leq \Cl{c18}(M)(\|\ad_1 - \ad_2 \|_\coz + \|\as_1 - \as_2 \|_\coz + \| v_1 - v_2\|_\coz).
	\end{equation}
	Then,
	\begin{align}
		\| \aos - \ats \|_\coz &= \| (\aos - \ats)(x, t) - (\aos - \ats)(x, 0) \|_\coz \nonumber\\
	&\leq T^{(1+\lambda)/2}\| \aos - \ats \|_{C^{1,(1+\lambda/2)}(\bar Q_T)} \nonumber\\
	&\leq T^{(1+\lambda)/2} \Cr{c18}(M)\(\|\ad_1 - \ad_2 \|_\coz + \|\as_1 - \as_2 \|_\coz + \| v_1 - v_2\|_\coz\). \label{eq:adiffbound}
	\end{align}
	If we take $T_0 = T$ such that
	$$ \max\{T \Cr{c12}(M),~ T^{(1+\lambda)/2} \Cr{c18}(M)\} < \frac{1}{3} $$
	we see by \eqref{eq:vdiffbound} and \eqref{eq:adiffbound} that $F$ is a contraction in $X_M$.	
	\paragraph{Conclusion and regularity.}
			According to the Banach fixed-point theorem $F$ has a unique fixed point $(\ad, \as, v)$, which together with $m$ from \eqref{eq:m_comp} is the unique solution of \eqref{eq:system_a}, \eqref{eq:bcs_a}. By \eqref{eq:m_comp} and \eqref{eq:a_in_sob} we have that
			$$m, \ad, \as \in \sobto.$$
			Due to \eqref{eq:vODEsimple}, \eqref{eq:vODEestder}, and \eqref{eq:vbound} we get
			$$ v \in C^{1,1}(\bar Q_T).$$
			By \eqref{eq:a_positive}, \eqref{eq:vstar_solution}, and \eqref{eq:m_positive}
			we get the non-negativity
			$$ \ad, \as, v, m \geq 0.$$
			
			Moreover we note that due to the non negativity of $v$, $0\leq v_0\leq 1$, and
			$$(1-v)_t \geq - \mu_v v (1-v), \quad (1-v)(\cdot, 0)\geq 0 $$
			$(1-v)$ can not become negative and hence $v \leq 1$.			
	\end{proof}
	\end{theorem}
%	\comm{jn @ nkls: Is the next theorem also valid for more than 2d?} \comm{Yes. More precisely both theorems use the Sobolev embedding A3, which is true for $d=1,2,3$. One could also adjust it easily to even higher dimensions if one would be interested in that.}
	
Our next result shows that the strong solutions which we constructed in Theorem \eqref{loc_existence} are indeed classical solutions.
%\comm{nks: usually strong solutions implies "smooth enough" and classical implies "infinitely smooth", but in the cases that I have in mind the distinction is clear. We should make it clear here as well.}
	\begin{theorem}[Regularity]\label{thm:reg_local}
		Under the initial and boundary conditions \eqref{eq:bcs_a} and \eqref{eq:iniData} the solution in Theorem \ref{loc_existence} satisfies %  for any $p>5$ and $\ell := 1 - \frac{5}{p}$
		\begin{equation}
			(\as, \ad, v, m) \in (C^{2+l,1+l/2}(\bar Q_{T_0}))^4,
		\end{equation}
		for $0<l<1$.
		\begin{proof}			
			\newcommand{\ddxi}{\partial_{x_i} }
			\newcommand{\ddxidxj}{\partial^2_{x_i, x_j}  }
			We use Theorem \ref{loc_existence} and the Sobolev embedding \ref{thm:sob_embed_hoelder}. Then we obtain for a sufficiently large $p>5$, that
			\begin{equation}
				\ad, \as, m \in \hoelohz{l}. \label{eq:am_in_sob}
			\end{equation}
			We further derive from \eqref{eq:system_a} that
			\begin{equation}
				(\ddxi v)_t = h_1 \ddxi v -h_2, \label{eq:dvdxiODE}
			\end{equation}
			where 
			\begin{align}
				h_1 &= -m + \mu_v \vfree - \mu_v v(1 + \chi_S e^{\chi_S v} a^S + \chi_D e^{\chi_D v} a^D ), \label{eq:help1dvdxi}\\
				h_2 &= v \ddxi m + \mu_v v (e^{\chi_S v} \ddxi a^S + e^{\chi_D v} \ddxi a^D). \label{eq:help2dvdxi}
			\end{align}
			Because of \eqref{eq:am_in_sob} and $v\in \coo$ we get
			\begin{equation}
				h_1, h_2 \in \hoelsz{l}{l/2}. \label{eq:helper_in_hoel}
			\end{equation} 
			The solution of \eqref{eq:dvdxiODE} is given by
			\begin{equation} \label{eq:dvdxi_sol}
			\ddxi v = \ddxi v_0(x) e^{\int_{0}^{t}h_1(x,s)\, ds} - \int_{0}^{t} h_2(x, \tau) e^{\int_{0}^{\tau}h_1(x,s)\, ds}\, d\tau,
			\end{equation}
			and hence by \eqref{eq:helper_in_hoel}
			\begin{equation}
				\ddxi v \in \hoelsz{l}{l/2}. \label{eq:dvdxi_in_hoel}
			\end{equation}
			The equation for $\ad$ in \eqref{eq:system_a} can be written as
			\begin{equation}
			\ad_t - \Delta \ad - \chi_D \nabla v \cdot \nabla \ad - h_3 \ad = h_4, \label{eq:ad_pde}
			\end{equation}
			where 
			\begin{align}
				h_3 &=  (\mu_D - \chi_D \mu_v v) \vfree \in \hoelsz{l}{l/2}\\
				h_4 &= \chi_D \ad v m - \memt a^D \in \hoelsz{l}{l/2}
			\end{align}
			by \eqref{eq:am_in_sob}, \eqref{eq:dvdxi_in_hoel}, and \eqref{eq:memt_lipschitz}. Thus, we can apply Theorem \ref{thm:parabolic_hoel} and get together with \eqref{eq:dvdxi_in_hoel} that the solution of \eqref{eq:ad_pde} satisfies
			\begin{equation}
				\ad \in \hoelsz{2 + l}{1 + l/2}. \label{eq:ad_high_reg}
			\end{equation}
			Similarly, the equation for $\as$ in \eqref{eq:system_a} can be rewritten as
			\begin{align}
			\as_t - \Delta \as - \chi_S \nabla v \cdot \nabla \as - h_5 \as = h_6 \label{eq:as_pde},\\
			h_5 = (\mu_S - \chi_S\mu_v v) \vfree \in \hoelsz{l}{l/2}, \\
			h_6 = \chi_S \as v m + \memt \ad \in \hoelsz{l}{l/2}.
			\end{align}
			Applying Theorem \ref{thm:parabolic_hoel} we obtain 
						\begin{equation}
						\as \in \hoelsz{2 + l}{1 + l/2}. \label{eq:as_high_reg}
						\end{equation}
			Furthermore, \eqref{eq:am_in_sob}, $v\in \hoels{1}{1}$, \eqref{eq:m_linpar1}, and \eqref{eq:m_linpar2} yield
			\begin{equation}
			m \in \hoelsz{2 + l}{1 + l/2}. \label{eq:m_high_reg}
			\end{equation}
			By using \eqref{eq:dvdxi_in_hoel} together with \eqref{eq:ad_high_reg}, \eqref{eq:as_high_reg}, and \eqref{eq:m_high_reg} and repeating the proof of \eqref{eq:dvdxi_in_hoel} for $\ddxidxj v$, we get
			\begin{equation}
				\ddxidxj v \in \hoelsz{l}{l/2}.  \label{eq:dvdxidxj_in_hoel}
			\end{equation}
			The equation for $v$ in \eqref{eq:system_a} provides further that
			$$ v_t = - mv + \mu_v v \vfree \in \hoelsz{2+l}{l/2},$$
			which yields together with $v\in \hoelsz{1}{1}$ and \eqref{eq:dvdxidxj_in_hoel} that
			$$ v \in \hoelsz{2+l}{1 + l/2}.$$
 		\end{proof}
	\end{theorem}
	\begin{remark} Let us note that the local existence of classical solutions that follow from the Theorems \ref{loc_existence} and \ref{thm:reg_local} is valid also for more than two space dimensions.
\end{remark}

%---------------------------------------------------------------------------------------
%---------------------------------------------------------------------------------------

\section{A-priori estimates for $\| \ad(\cdot,t)\|_\linfo, \| \as(\cdot,t)\|_\linfo$}\label{sec:apriori}
To extend the local (in time) solutions whose existence we have established in the last section to global (in time) solutions we need some {\it a priori} estimates.
Establishing those estimates is the purpose of this section. Let $(\ad,\as,v,m)\in \(\mathcal C^{2,1} (Q_T)\)^4$ be a classical solution of  \eqref{eq:system_a} in $[0,T]$ for any $T>0$. In what follows we will show the corresponding a priori estimates. 
We begin by proving $\| \cdot \|_{L^1(\Omega)}$ bounds for $\ad$, $\as$ and $m$ uniformly in time.
\begin{lem} \label{thm:l1_bounds}
Let $(\ad,\as,v,m)\in \(\mathcal C^{2,1} (Q_T)\)^4$ be a solution of  \eqref{eq:system_a}, then we have for all $t\in (0,T)$,
\begin{subequations}
	\begin{align}
		\| \ad(\cdot,t) \|_{L^1(\Omega)}\ \leq\ \|\cd(\cdot,t) \|_{L^1(\Omega)}\ \leq\  \max\left\{\|\cd_0\|_{L^1(\Omega)}, |\Omega|\right\} \label{lemma.1}\\
		\| \as(\cdot,t) \|_{L^1(\Omega)}\ \leq\  \|\cs(\cdot,t) \|_{L^1(\Omega)}\ \leq\ \max\left\{\|\cs_0\|_{L^1(\Omega)}, {\cs_{\max}}  \right\} \label{lemma.2} \\
		\|m(\cdot,t)\|_{L^1(\Omega)}  \ \leq \max \left\{ \|m_0\|_{L^1(\Omega)}, \max\left\{\|\cd_0\|_{L^1(\Omega)},|\Omega|\right\} + \max\left\{\|\cs_0\|_{L^1(\Omega)}, \cs_{\max}  \right\}  \right\}\label{lemma.3} 
	\end{align}
	with 
	\[ \cs_{\max} := \frac{|\Omega|}{2}\( 1 + \sqrt{1 + 4\frac{\mu_M}{\mu_S|\Omega|} \max\left\{\|\cd_0\|_{L^1(\Omega)},|\Omega|\right\} } \) .\]
%	\comm{jn: Why is $\rho_{\text{max}}$ called  $\rho_{\text{max}}$?} \comm{$\cs_{\max}$}
\end{subequations}
\end{lem}
\begin{proof}
	We integrate the $\cd$ equation in \eqref{eq:reduced_II} over $\Omega$ and employ the boundary conditions \eqref{eq:bcs_a} and $\cd \geq 0$:
	\begin{align*}
		\frac{d}{dt}\| \cd(\cdot,t) \|_{L^1(\Omega)}=& - \|\memt \cd(\cdot,t) \|_{L^1(\Omega)} 
			+\mu_D \| \cd(\cdot,t) \|_{L^1(\Omega)} - \mu_D \int_{\Omega}\cd(x,t) \cs(x,t) dx \\
			&-\mu_D \int_{\Omega}(\cd(x,t))^2dx -\mu_D \int_{\Omega} \cd(x,t) v(x,t)dx.
	\end{align*}
	Due to the positivity of $\cd, \cs$ and $v$ we obtain
	\begin{align*}
		\frac{d}{dt}\| \cd(\cdot,t) \|_{L^1(\Omega)}\ \leq\ & 
			\mu_D \| \cd(\cdot,t) \|_{L^1(\Omega)} - \mu_D \int_{\Omega}(\cd(x,t))^2 dx.
	\end{align*}
	or, after the boundedness of $\Omega$ and the corresponding embeddings, as 
	\[
		\frac{d}{dt}\| \cd(\cdot,t) \|_{L^1(\Omega)}\  \leq 
			\ \mu_D \| \cd(\cdot,t) \|_{L^1(\Omega)} - \frac{\mu_D}{|\Omega|} \| \cd(\cdot,t) \|_{L^1(\Omega)}^2.
	\]
	Since the right hand side is a quadratic polynomial with roots $0$ and $|\Omega|$, we deduce by comparison
	\[
		\| \cd(\cdot,t) \|_{L^1(\Omega)}\ \ \leq\ \max\left\{\|\cd_0\|_{L^1(\Omega)},|\Omega|\right\}.
	\]
				
	Similarly, we see that due to the positivity of $\cs$, $v$, the $\cs$ equation \eqref{eq:reduced_II} implies
	\begin{align*}
		\frac{d}{dt}\| \cs(\cdot,t) \|_{L^1(\Omega)}
			=&\ \| \memt \|_{L^\infty(\Omega)} \| \cd(\cdot,t) \|_{L^1(\Omega)} +\mu_S \| \cs(\cdot,t) \|_{L^1(\Omega)} - \mu_S \int_{\Omega}\cs(x,t) \cd(x,t) dx\\
			&-\mu_S \int_{\Omega}(\cs(x,t))^2dx -\mu_S \int_{\Omega} \cs(x,t) v(x,t)dx\\
			\leq&\ \mu_M \| \cd(\cdot,t) \|_{L^1(\Omega)} + \mu_S \| \cs(\cdot,t) \|_{L^1(\Omega)} -\mu_S \int_{\Omega}(\cs(x,t))^2 dx\\
			\leq&\ \mu_M \| \cd(\cdot,t) \|_{L^1(\Omega)} + \mu_S \| \cs(\cdot,t) \|_{L^1(\Omega)} -\frac{\mu_S}{|\Omega|} \| \cs(\cdot,t) \|_{L^1(\Omega)}^2\\
			\leq&\ \mu_M \max\left\{\|\cd_0\|_{L^1(\Omega)},|\Omega|\right\} + \mu_S \| \cs(\cdot,t) \|_{L^1(\Omega)} -\frac{\mu_S}{|\Omega|} \| \cs(\cdot,t) \|_{L^1(\Omega)}^2.
	\end{align*}		
	The right-hand side has two roots, one negative and one positive that is larger than $|\Omega|$:
	\[
		%\frac{|\Omega|}{2}\( 1 + \sqrt{1 + 4\frac{\memt}{\mu_S|\Omega|} \max\left\{\|\cd_0\|_{L^1(\Omega)},|\Omega|\right\} } \) \leq
		\frac{|\Omega|}{2}\( 1 + \sqrt{1 + 4\frac{\mu_M}{\mu_S|\Omega|} \max\left\{\|\cd_0\|_{L^1(\Omega)},|\Omega|\right\} } \) = \cs_{\max}.
	\]
	We deduce by comparison 
	\[
		\| \cs(\cdot,t) \|_{L^1(\Omega)} \leq \max\left\{\|\cs_0\|_{L^1(\Omega)}, {\cs_{\max}}  \right\}.
	\]
	For $m$ we get from \eqref{eq:reduced_II}, after integration over $\Omega$, due to the positivity of $\cd$, $\cs$, $m$, and the boundary conditions \eqref{eq:bcs_a}, that:
	\[ 
		\frac{d}{dt} \|m(\cdot,t)\|_{L^1(\Omega)}\leq  \|\cd(\cdot,t)\|_{L^1(\Omega)} + \|\cs(\cdot,t)\|_{L^1(\Omega)} 
		-\|m(\cdot,t)\|_{L^1(\Omega)}. 
	\]
	Using \eqref{lemma.1} and \eqref{lemma.2} we obtain
	\[
		 \frac{d}{dt} \|m(\cdot,t)\|_{L^1(\Omega)}\leq  \max\left\{\|\cd_0\|_{L^1(\Omega)},|\Omega|\right\} + \max\left\{\|\cs_0\|_{L^1(\Omega)}, {\cs_{\max}}  \right\} 
		-\|m(\cdot,t)\|_{L^1(\Omega)}. 
	\]
	Finally we deduce that 
	\[ 
		\|m(\cdot,t)\|_{L^1(\Omega)} \leq \max \left\{ \|m_0\|_{L^1(\Omega)}, \max\left\{\|\cd_0\|_{L^1(\Omega)},|\Omega|\right\} + \max\left\{\|\cs_0\|_{L^1(\Omega)}, {\cs_{\max}}  \right\}  \right\}. 
	\]
\end{proof}

We have shown uniform in time $L^1$ bounds of $\ad, \as, m.$ In order to prove a uniform in time $L^\infty$ estimate for $a$ we need the following Lemma which can be found (for an arbitrary number of dimensions) in \cite[Lemma 1]{kowalczyk2008global}
and is an extension of  \cite[Lemma 4.1]{horstmann2005boundedness}

\begin{lem}\label{thm:lemma-3.2.Tao}
	Let $m_0\in W^1_\infty(\Omega)$ and let $\cd, \cs, m$ satisfy the equation $m$ in \eqref{eq:reduced_II} together with $\frac{\partial m}{\partial \nu}\big |_{\Gamma_T}=0$. 
	Moreover, we assume that $\|\cd(t)+\cs(t)\|_{L^\rho(\Omega)} \leq \C$ %C_{2\_1}$
	for $1\leq \rho$ and all $t\in(0,T)$.
	Then for $\rho < 2$
	\begin{equation}\label{eq:21.4}
		\|m(t)\|_{W^1_q(\Omega)}\leq \C(q)%C_{2\_2}(q)
		                                              ,\quad t\in(0,T),
	\end{equation}
	where
	\begin{equation}\label{eq:21.5}
		q<\frac{2\rho}{2-\rho}.
	\end{equation}
	Moreover, if $\rho=2$ then \eqref{eq:21.4} is valid for $q<+\infty$, if $\rho>2$ then \eqref{eq:21.4} is valid for $q=+\infty$.
\end{lem}

\begin{proof}
	See Appendix \ref{sec:proof_of_lemma-3.2.Tao}.
\end{proof}

We now combine Lemma \ref{thm:lemma-3.2.Tao} with a suitable Sobolev embedding to obtain a uniform bound for $m$ in higher Lebesgue spaces:
	
\begin{lem}\label{thm:lemma-3.3-Tao}
		Let $m_0\in W^1_\infty(\Omega)$, and $\cd, \cs, m$ satisfy the equation for $m$ in \eqref{eq:reduced_II} together with $\frac{\partial m}{\partial \nu}\big |_{\Gamma_T}=0$.
		Moreover, we assume that $\|\cd(t)+\cs(t)\|_{L^\rho(\Omega)} \leq \C$ for $1\leq \rho$, and all $t\in(0,T)$.
		Then, 
		\begin{equation}\label{eq:21.4b}
		\|m(t)\|_{L_r(\Omega)}\leq \C(q)\quad t\in(0,T),
		\end{equation}
		for any $r>\rho$ that satisfies
		\begin{equation}
			\frac 1 r + 1 > \frac 1 \rho .
		\end{equation}
\end{lem}
	
\begin{proof}
%	\comm{jan@nks: shorten this proof}\\
	The proof is based on the Sobolev embedding $W^1_q(\Omega) \hookrightarrow  L^{r'}(\Omega)$   for $r'< \frac{2q}{2-q}$, and Lemma \ref{thm:lemma-3.2.Tao}.
	
	Since  $q<\frac{2\rho}{2-\rho}$, it holds that $2r'< (2+r')q<(2+r')\frac{2\rho}{2-\rho}$. That is, $\( 2-\frac{2\rho}{2-\rho}\)r'<\frac{4\rho}{2-\rho}$ or 
	\begin{equation}
		\frac{1}{\rho}-1<\frac{1}{r'}.
	\end{equation}
		
	Then it holds 
	\begin{equation}\label{eq:Tao.3.6}
		\|m(t)\|_{L^{r'}(\Omega)}\leq \C,\quad t\in(0,T),
	\end{equation}
	where $r'>\rho$ such that 
	\begin{equation}\label{eq:Tao3.7}
		\frac{1}{r'}+1>\frac{1}{\rho}.
	\end{equation}
\end{proof}

The main result of this section is the following theorem which asserts uniform in time a priori bounds for $\ad$ and $\as$ in $\|\cdot\|_\linfo$.
\begin{theorem}
Let $(\ad,\as,v,m)\in \(\mathcal C^{2,1}(Q_T)\)^4$ be a solution of   \eqref{eq:system_a}, and
let \eqref{eq:mud_bound} hold. Then for all $t\in (0,T)$:
\begin{equation} \label{eq:a_linfty_bounds}
	\| \ad(\cdot,t)\|_\linfo, \| \as(\cdot,t)\|_\linfo \leq \C.
\end{equation}
\end{theorem}
\begin{proof}
	
	The proof is divided into 4 steps. We first derive a basic estimate, prove $L^p$ bounds for all $p$ in step two and three and finally prove the $\linfo$ estimate.
	\paragraph{Step 1: First $L^p(\Omega)$ estimates.}  We set $\gamma = 0$ if $p\leq 2$ and $\gamma \in (0,1)$ otherwise, and
	$a_\gamma = a+ \gamma \geq \gamma \geq 0$ so that
	\begin{equation}\label{eq:ag_gradient}
		\nabla \(\adgph\) = \frac p 2  (\adg)^{p/2 -1} \nabla \adg, \quad\nabla \(\asgph\) = \frac p 2  (\asg)^{p/2 -1} \nabla \asg, \quad \text{for any } p>1.
	\end{equation}
	Since $0\leq v \leq 1$ we can consider the integrals $\int_{\Omega} e^{\chi_D v} \adgp dx$, $\int_{\Omega} e^{\chi_S v} \asgp dx$ instead of $\int_{\Omega}\adgp dx,$ $\int_{\Omega}\asgp dx$, and get moreover
	\begin{equation}
		0 \leq \mu_D - \chi_D \mu_v \leq \mu_D, \quad 0 \leq \mu_S - \chi_S \mu_v \leq \mu_S, \label{eq_a_estimate_condition}
	\end{equation}
	using the above assumption.
	Using \eqref{eq:system_a}, \eqref{eq_a_estimate_condition}, partial integration, \eqref{eq:ag_gradient}, \eqref{eq:memt_bound}, and the fact that $0\leq v \leq 1$, we obtain 
	\begin{align}
		\frac{d}{dt}\int_{\Omega} \, e^{\chi_D v} \adgp dx =& \int_{\Omega} \chi_D e^{\chi_D  v} \partial_t v \adgp \, dx + 
		\int_{\Omega} e^{\chi_D \,v } p (\adg)^{p-1} \partial_t \ad \, dx \nonumber \\
		 =& - \chi_D \int_{\Omega} e^{\chi_D v} m v \adgp dx + \chi_D\, \mu_v \int_{\Omega} e^{\chi_D  v}p \adgp v \vfree \, dx \nonumber\\ &+
		\int_{\Omega} p (\adg)^{p-1} \nabla \cdot(e^{\chi_D v}\nabla \ad)\,dx  
		+ \chi_D \int_{\Omega}e^{\chi_D  v} p (\adg)^{p-1} \ad v m \, dx \nonumber\\ 
		&- \int_{\Omega}\memt\, \ad e^{\chi_D v}p (\adg)^{p-1}\, dx 
		+ \int_\Omega e^{\chi_D v} p (\adg)^{p-1} (\mu_D - \chi_D\, \mu_v\, v) \ad \vfree\, dx \nonumber\\
		\leq &(\mu_D \, p + \chi_D \, \mu_v\, p) \int_{\Omega} e^{\chi_D  v} \adgp \, dx + \chi_D \, p \int_{\Omega} e^{\chi_D \, v} \adgp m \, dx \nonumber\\
		&- \int_{\Omega}p(p-1)(\adg)^{p-2}|\nabla \adg|^2 e^{\chi_D v}\, dx \nonumber \\
		\leq& - \frac{4(p-1)}{p}  \int_{\Omega} |\nabla \adgph|^2\, dx + (\mu_D\, p + \chi_D \, \mu_v\, p)e^{\chi_D} \int_{\Omega}  \adgp \, dx \nonumber \\ &+ \chi_D \, p e^{\chi_D} \int_{\Omega} m \adgp \, dx. \label{eq:ad_lp_estimate}
	\end{align}
	Similarly, we get
		\begin{align}
		\frac{d}{dt}\int_{\Omega} \, e^{\chi_S v} \asgp dx =& \int_{\Omega} \chi_S e^{\chi_S  v} \partial_t v \asgp \, dx + 
		\int_{\Omega} e^{\chi_S \,v } p (\asg)^{p-1} \partial_t \as \, dx \nonumber \\
		=& - \chi_S \int_{\Omega} e^{\chi_S v} m v \asgp dx + \chi_S\, \mu_v \int_{\Omega} e^{\chi_S  v}p \asgp v \vfree \, dx \nonumber \\
		 &+\int_{\Omega} p (\asg)^{p-1} \nabla \cdot(e^{\chi_S v}\nabla \as)\,dx 
		+ \chi_S \int_{\Omega}e^{\chi_S  v} p (\asg)^{p-1} \as v m \, dx \nonumber \\ 
		&+ \int_{\Omega}\memt\, \ad e^{\chi_S v}p (\asg)^{p-1}\, dx \nonumber\\
		&+ \int_\Omega e^{\chi_S v} p (\asg)^{p-1} (\mu_S - \chi_S\, \mu_v\, v) \as \vfree\, dx \nonumber\\
		\leq& - \frac{4(p-1)}{p}  \int_{\Omega} |\nabla \asgph|^2\, dx + (\mu_S\, p + \chi_S \, \mu_v\, p)e^{\chi_S} \int_{\Omega}  \asgp \, dx \nonumber \\
		&+ \chi_S \, p e^{\chi_S} \int_{\Omega} m \asgp \, dx + \mu_M\, p e^{\chi_S} \int_{\Omega} \ad  (\asg)^{p-1}\, dx. \label{eq:as_lp_estimate}
		\end{align}

%------------------------------------------------------------------------------------------------
	\paragraph{Step 2: Raise of $p$.} We assume that both $\| \adg(\cdot,t) \|_\lqn, \| \asg(\cdot,t) \|_\lqn \leq \Cl{c1}$ for some $q \geq 1$ and show that $$\| \adg(\cdot,t)\|_{L^p(\Omega)}, \| \asg(\cdot,t)\|_{L^p(\Omega)}\leq \Cl{c2},$$
	where $ p = \frac{4 q}{3}$.
	
	Since we are in  $d = 2$ space dimensions the inequality
	\begin{equation}
		\frac{dp}{dp + 2 q} < 1 + \frac{2}{d} - \frac{1}{q},
	\end{equation}
	is true and allows us to find $r>1$, such that
		\begin{equation}
		\frac{dp}{dp + 2 q} < \frac{1}{r}<  1 + \frac{2}{d} - \frac{1}{q}.
		\end{equation}
		The first inequality justifies the Gagliardo-Nirenberg inequality
		\begin{equation} \label{eq:gag_nir_1}
			\| \cdot\|^{2r}_{L^{2r}} \leq \C \| \cdot\|^{2(r-1)}_{L^{2q/p}} \| \cdot \|^2_{W^1_2},
		\end{equation}
			and due to the second inequality there is a dual exponent $r'$ of $r$ that satisfies the conditions of Lemma \ref{thm:lemma-3.3-Tao}. We take $a \in \{\adg, \asg\}$.
			Applying Young's inequality, \eqref{eq:gag_nir_1}, Lemma \ref{thm:lemma-3.3-Tao}, and assumption $\| a(\cdot,t) \|_\lqn \leq \Cr{c1}$, we get for any $\varepsilon>0$
		\begin{align}
			\int_{\Omega} m a^p \, dx &\leq \C(\varepsilon) \int_{\Omega} m^{r'}\, dx + \varepsilon \int_{\Omega} a^{pr} \, dx \nonumber\\ & \leq \Cl{c3}(\varepsilon) + \varepsilon \| a^{p/2}\|^{2r}_{L^{2r}} \nonumber \\
			& \leq \Cr{c3}(\varepsilon) + \varepsilon \C \| a \|^{p(r-1)}_{L^q} \|a^{p/2}\|^2_{W^1_{2}} \nonumber \\
			&\leq  \Cr{c3}(\varepsilon) + \varepsilon \Cl{c4} \int_{\Omega}a^p \, dx + \varepsilon \Cr{c4}  \int_{\Omega}|\nabla a^{p/2}|^2 \, dx. \label{eq:map_estimate}
		\end{align}
		Since we are in two space dimensions we have the Gagliardo-Nirenberg interpolation inequality
		\begin{equation}\label{eq:gag_nir_2}
				\| \cdot\|_{L^{2}} \leq \C \| \cdot\|^{1/4}_{W^1_2} \| \cdot \|^{3/4}_{L^{3/2}},
		\end{equation}
		and we can moreover estimate $\int_\Omega a^p\, dx$ by employing \eqref{eq:gag_nir_2}, Young's inequality and $\| a(\cdot,t) \|_\lqn \leq \Cr{c1}$
			\begin{align}
				(\Cl{c6} + \beta) \int_{\Omega} a^p \, dx & = (\Cr{c6} + \beta) \int_{\Omega} (a^{p/2})^2 \, dx \nonumber \\
				& = (\Cr{c6} + \beta) \|a^{p/2} \|_{L^2}^2 \nonumber \\
				& \leq \C (\Cr{c6} + \beta) \| a^{p/2} \|^{1/2}_{W_2^1(\Omega)} \| a^{p/2} \|^{3/2}_{L^{3/2}(\Omega)} \nonumber \\
				& \leq \frac{\beta}{2} \|a^{p/2} \|^2_{W_2^1(\Omega)} + \Cl{c7} \| a^{p/2} \|^{2}_{L^{3/2}(\Omega)} \nonumber \\
				& = \frac{\beta}{2} \| a^{p/2} \|^2_{W_2^1(\Omega)} + \Cr{c7} \| a \|^{p}_{L^{3p/4}(\Omega)} \nonumber \\
				& \leq \frac{\beta}{2} \| a^{p/2} \|^2_{W_2^1(\Omega)} + \C, \label{eq:ap_sob_estimate}
			\end{align}
		where $\Cr{c6}$ and $\beta$ are arbitrary positive numbers.
		
		In order to prove the $L^p$ bound for $\ad$ we insert \eqref{eq:map_estimate} where $a=\ad$ into \eqref{eq:ad_lp_estimate} and fix  $\varepsilon$ such that
		$\varepsilon\chi_D p e^{\chi_D} \Cr{c4} <2(p-1)/p $	to obtain
		
		\begin{equation}
			\frac{d}{dt} \int_{\Omega} e^{\chi v} \adgp\, dx \leq - \frac{2(p-1)}{p}\int_{\Omega} |\nabla \adgph|^2 \, dx
			+ \Cl{c9} \int_{\Omega} \adgp \, dx + \Cl{c10}. \label{eq:ad_lp_estimate_2}
		\end{equation}
		By adding $\beta \int_{\Omega} \adgp \, dx$ on both sides of \eqref{eq:ad_lp_estimate_2} we get
				\begin{equation}
				\frac{d}{dt} \int_{\Omega} e^{\chi_D v} \adgp\, dx + \beta \int_{\Omega} \adgp \, dx \leq - \frac{2(p-1)}{p}\int_{\Omega} |\nabla \adgph|^2 \, dx
				+ (\Cr{c9} + \beta) \int_{\Omega} \adgp \, dx + \Cr{c10}. \label{eq:ad_lp_estimate_3}
				\end{equation} 
		We can now insert \eqref{eq:ap_sob_estimate}, where $a=\ad$ and $\beta =2(p-1)/p$ into \eqref{eq:ad_lp_estimate_3} and get
						\begin{equation}
						\frac{d}{dt} \int_{\Omega} e^{\chi_D v} \adgp\, dx + \frac{p-1}{p} \int_{\Omega} \adgp \, dx \leq  \Cl{c11}, \label{eq:ad_lp_estimate_4}
						\end{equation} 
						which implies
						\begin{equation}
						\frac{d}{dt} \int_{\Omega} e^{\chi_D v} \adgp\, dx + \frac{p-1}{p\,e^{\chi_D}} \int_{\Omega} e^{\chi_D v} \adgp \, dx \leq  \Cr{c11}, \label{eq:ad_lp_estimate_5}
						\end{equation} 
						and thus
						\begin{equation}
						\int_{\Omega} e^{\chi_D v} \adgp\, dx \leq  \C. \label{eq:ad_lp_estimate_6}
						\end{equation}
						Hence we have shown that
						\begin{equation}
						\|\ad(\cdot,t) \|_\lpn \leq  \C. \label{eq:ad_lp_estimate_final}
						\end{equation}
						
				An application of Young's inequality and \eqref{eq:ad_lp_estimate_final} lead to
				\begin{align}
				\int_{\Omega} \ad  (\asg)^{p-1}\, dx &\leq \frac{p-1}{p} \int_{\Omega} \asgp \, dx + \frac{1}{p} \int_{\Omega} (\ad)^p \, dx \nonumber\\
				&\leq \C \int_{\Omega} \asgp\, dx + \C. \label{eq:ad_asp_estimate}
				\end{align}
				Inserting \eqref{eq:ad_asp_estimate} into \eqref{eq:as_lp_estimate} yields
				\begin{equation}
					\frac{d}{dt}\int_{\Omega} \, e^{\chi_S v} \asgp dx 
					\leq - \frac{4(p-1)}{p}  \int_{\Omega} |\nabla \asgph|^2\, dx + \C \int_{\Omega}  \asgp \, dx 
					+ \C \int_{\Omega} m \asgp \, dx + \C. \label{eq:as_lp_estimate_2}
				\end{equation}
				Since \eqref{eq:map_estimate} and \eqref{eq:ap_sob_estimate} are also valid for $a=\as$ we can repeat the steps
				in \eqref{eq:ad_lp_estimate_2}--\eqref{eq:ad_lp_estimate_6} for \eqref{eq:as_lp_estimate_2} to get 
				\begin{equation}
				\|\as(\cdot,t) \|_\lpn \leq  \C. \label{eq:as_lp_estimate_final}
				\end{equation}
				
%-----------------------------------------------------------------------------------------
				\paragraph{Step 3: $L^p$ bounds for all $p  \geq 1$.}
				From Lemma \ref{thm:l1_bounds} and the previous step,
				\begin{equation}
					\|\ad(\cdot,t) \|_{L^{(4/3)^n}(\Omega)},~\|\as(\cdot,t) \|_{L^{(4/3)^n}(\Omega)} \leq \C(n)< \infty\quad \text{for all }n\in\mathbb{N},
				\end{equation} 
				follows from induction. Hence, we have that
				\begin{equation}
						\|\ad(\cdot,t)\|_\lpn,~\|\as(\cdot,t) \|_\lpn \leq \C(p)< \infty \quad \text{for all }p \geq 1.
						\label{eq:ad_as_all_lp}
				\end{equation}
				
%---------------------------------------------------------------------------------------				
				\paragraph{Step 4: $L^\infty$ bounds.} 
				For the step we employ this technique used in \cite{Alikakos.1979} and applied in the case of KS system in \cite{Corrias.2004}.	We are in $d=2$ space dimensions and we know from step 3 
				that there is $\rho > d =2$ such that
				$\| \cd + \cs\|_{L^\rho(\Omega)} \leq C_{20}$. Therefore we get by Lemma \ref{thm:lemma-3.2.Tao}
				\begin{equation} \label{eq:mlinf_bound}
					\|m\|_\linf \leq \C.
				\end{equation}
				Inserting \eqref{eq:mlinf_bound} back into \eqref{eq:ad_lp_estimate} we get that
			\begin{equation}
		\frac{d}{dt} \int_{\Omega} e^{\chi_D v} (\ad)^p \, dx  + \frac{4(p-1)}{p}\int_{\Omega} |\nabla (\ad)^{p/2}|^2 \, dx
		\leq  \Cl{c22} \,p \int_{\Omega} (\ad)^p \, dx \quad \text{for all }p\geq \rho.\label{eq:ad_lp_to_linf_1}
			\end{equation}
			We define the sequence $p_k = 2^k,~ k\in \mathbb{N}$ and moreover, we apply the Gagliardo-Nirenberg inequality
					\begin{equation}\label{eq:gag_nir_3}
					\| \cdot\|_{L^{2}} \leq \C \| \cdot\|^{1/2}_{W^1_2} \| \cdot \|^{1/2}_{L^{1}}.
					\end{equation}
			 Thus, we get for  $a \in \{\ad, ~ \as \}$ by \eqref{eq:gag_nir_3} and Young's inequality that
			\begin{equation}
				\int_{\Omega} a^{p_{k}} \, dx = \| a^{p_{k-1}} \|_{L^2(\Omega)}^2 \leq \Cl{c23} \| a^{p_{k-1}} \|_{W^1_2 (\Omega)} \| a^{p_{k-1}} \|_{L^1(\Omega)} \leq \Cr{c23} \(\frac{1}{\varepsilon_k}  \| a^{p_{k-1}} \|^2_{L^1(\Omega)} + \varepsilon_k \| a^{p_{k-1}} \|^2_{W^1_2 (\Omega)}\)
			\end{equation}
			which implies for sufficiently small $\varepsilon_k$
			\begin{equation}
					\int_{\Omega} a^{p_{k}} \, dx \leq \Cr{c23} \(\frac{1}{\varepsilon_k}  \| a^{p_{k-1}} \|^2_{L^1(\Omega)} + \varepsilon_k \|\nabla  a^{p_{k-1}} \|^2_{L^2(\Omega)}\). \label{eq:ap_ph_estimate}
			\end{equation}
			Adding $\varepsilon_k e^{\chi_D} \int_{\Omega} (\ad)^{p_k} \, dx$ on both sides of \eqref{eq:ad_lp_to_linf_1}, choosing $\varepsilon_k$ such that 
			\begin{equation}
				(\Cr{c22}\, p_k + \varepsilon_k e^{\chi_D}) \Cr{c23} \,\varepsilon_k \leq 4(p_k-1)/p_k < 4 \label{eq:ad_linf_eps_choice}
			\end{equation}			
			 in \eqref{eq:ap_ph_estimate} for $a = \ad$ and inserting in \eqref{eq:ad_lp_to_linf_1} yield for $k\geq2$
						\begin{equation}
						\frac{d}{dt} \int_{\Omega} e^{\chi_D v} (\ad)^{p_k} \, dx  
						+ \varepsilon_k e^{\chi_D} \int_{\Omega} (\ad)^{p_k} \, dx \leq  \frac{(\Cr{c22}\, p_k + \varepsilon_k e^{\chi_D}) \Cr{c23}}{\varepsilon_k} \( \int_{\Omega} (\ad)^{p_{k-1}} \, dx \)^2. \label{eq:ad_lp_to_linf_2}
						\end{equation}
						The later implies that 
						\begin{equation}
						\frac{d}{dt} \int_{\Omega} e^{\chi_D v} (\ad)^{p_k} \, dx  
						\leq  -\varepsilon_k \int_{\Omega} e^{\chi_D v} (\ad)^{p_k} \, dx + \frac{(\Cr{c22}\, p_k + \varepsilon_k e^{\chi_D}) \Cr{c23}}{\varepsilon_k} M_{k-1}^2 , \label{eq:ad_lp_to_linf_3}
						\end{equation}
						where
						\begin{equation}
							M_k = \max \left\{ 1, \sup_{0< t < T}\int_{\Omega} e^{\chi_D v} (\ad)^{p_{k}} \, dx \right\}.
						\end{equation}
						By Gronwall's lemma we get from \eqref{eq:ad_lp_to_linf_3}, that
						\begin{equation}
							\int_{\Omega} e^{\chi_D v} (\ad)^{p_k} \, dx  \leq \max \left\{ \int_{\Omega} e^{\chi_D v_0} (\ad_0)^{p_k} \, dx, \frac{(\Cr{c22}\, p_k + \varepsilon_k e^{\chi_D}) \Cr{c23}}{\varepsilon_k^2} M_{k-1}^2 \right\}\quad \text{for }k\geq 2. \label{eq:ad_lp_to_linf_gronwall}
						\end{equation} 
						Hence 
						\begin{equation}
						M_k  \leq \max \left\{1,~  |\Omega| e^{\chi_D} \|(\ad_0)\|_\linfo^{p_k} ,~ \delta_k M_{k-1}^2 \right\}\quad \text{for }k\geq 2, \label{eq:ad_lp_to_linf_4}
						\end{equation}
						where $\delta_k = \max \{ 1,~(\Cr{c22}\, p_k + \varepsilon_k e^{\chi_D}) \Cr{c23}/\varepsilon_k^2$\}. 
% 						\comm{jn: I think it should be $\varepsilon_k$ instead of $\varepsilon_k^2$. -- nkls: it's correct here but it was wrong in \eqref{eq:ad_lp_to_linf_gronwall} (corrected it) --
% 						jn: Why is it  $\varepsilon_k^2$ in \eqref{eq:ad_lp_to_linf_gronwall}? -- nkls: Because the $\varepsilon_k^2$ term is the root in \eqref{eq:ad_lp_to_linf_3}-- jn: I agree.}
						Note that by \eqref{eq:iniData} and \eqref{eq:ad_as_all_lp}
						%\begin{equation}
						%	\|(\ad_0)\|_\linfo \C, \quad M_1 \leq \C, \label{eq:ad_lp_to_linf_5}
						%\end{equation}
						we can find a constant $\Cl{cspec}$ such that
						\begin{equation}
							M_1 + 1 \leq \Cr{cspec}, \quad	|\Omega| e^{\chi_D} \|(\ad_0)\|_\linfo^{p_k} \leq \Cr{cspec}^{p_k} \text{ for }k\geq 1. \label{eq:ad_const_choice}
						\end{equation}
						From \eqref{eq:ad_lp_to_linf_4}, \eqref{eq:ad_const_choice} and $\delta_k \geq 1$ we get that
						\begin{equation}
							M_k \leq \delta_k \, \delta_{k-1}^{p_1}\, \delta_{k-2}^{p_2} \cdots \delta_2^{p_{k-2}} \,
							\delta_1^{p_{k-1}} \Cr{cspec}^{p_k}.
						\end{equation}
						Furthermore, we get from \eqref{eq:ad_linf_eps_choice} that $\varepsilon_k$ can be chosen as $\varepsilon_k = \Cl{c26}/p_k$,
						where the constant $\Cr{c26}$ is independent of $k$. This yields 
						$$ \delta_k \leq \Cl{c27}p_k^3$$ and hence
						\begin{equation}
						M_k^{1/p_k} \leq \Cr{c27}^{\sum_{i=0}^{k-1} 2^{i-k}} \, 2^{3 \sum_{i=0}^{k-1} 2^{i-k} (k-i)} \Cr{cspec}
						\leq \Cr{c27}^{1-\frac{1}{p_k}} 2^{3 \sum_{i=1}^{k} \frac{i}{2^i}} \Cr{cspec}. \label{eq:ad_lp_to_linf_5}
						\end{equation}
						For $0< t< T$ we note that $\max\{1,~\|\ad(\cdot, t)\|_{L^{p_k}(\Omega)}\}\leq M_k^{1/p_k}$  by $0\leq v\leq 1$ and when taking $k \rightarrow \infty$ in \eqref{eq:ad_lp_to_linf_5} we eventually get
						\begin{equation}
							\|\ad(\cdot, t) \|_\linfo \leq \C. \label{eq:ad_lp_to_linf_final}
						\end{equation}
						
						Using the bounds \eqref{eq:mlinf_bound}, \eqref{eq:ad_lp_to_linf_final} as well as the sequence $p_k = 2^k,~ k\in \mathbb{N}$ in \eqref{eq:as_lp_estimate}
						yields for $k\geq 2$
						\begin{equation}
						\frac{d}{dt} \int_{\Omega} e^{\chi_D v} (\as)^{p_k} \, dx  + \frac{4(p_k-1)}{p_k}\int_{\Omega} |\nabla (\as)^{p_{k-1}}|^2 \, dx
						\leq  \C \,p_k \int_{\Omega} (\as)^{p_k} \, dx + \C \,p_k \int_{\Omega} (\as)^{p_k-1} \, dx.\label{eq:as_lp_to_linf_1}
						\end{equation}
						By Hölder's inequality we estimate
						\begin{equation}
							\int_{\Omega}(\as)^{p_k-1} \, dx \leq | \Omega|^{1/p_k} \(\int_{\Omega}(\as)^{p_k} \, dx \)^{(p_k-1)/p_k} \leq \C \( \int_{\Omega}(\as)^{p_k} \, dx + 1 \)
						\end{equation}
						and get 
						\begin{equation}
						\frac{d}{dt} \int_{\Omega} e^{\chi_D v} (\as)^{p_k} \, dx  + \frac{4(p_k-1)}{p_k}\int_{\Omega} |\nabla (\as)^{p_{k-1}}|^2 \, dx
						\leq  \Cl{c32}\,p_k \( \int_{\Omega}(\as)^{p_k} \, dx + 1 \) .\label{eq:as_lp_to_linf_2}
						\end{equation}
						
We add again $\varepsilon_k e^{\chi_S} \int_{\Omega} (\as)^{p_k} \, dx$ on both sides of \eqref{eq:as_lp_to_linf_2} and choose $\varepsilon_k$ such that 
\begin{equation}
(\Cr{c32}\, p_k + \varepsilon_k e^{\chi_S}) \Cr{c23} \,\varepsilon_k \leq 4(p_k-1)/p_k < 4 \label{eq:as_linf_eps_choice},
\end{equation} 
where  $\Cr{c23}$, and $\varepsilon_k$ are chosen such that \eqref{eq:ap_ph_estimate} is true for $a = \as$.
By setting $\varepsilon_k= \C / p_k$ we
find a constant $\Cl{c34} > \Cr{c23}$ such that
\begin{equation}
	\frac{(\Cr{c32}\, p_k + \varepsilon_k e^{\chi_S}) \Cr{c34}}{\varepsilon_k} \geq \Cr{c32} p_k. \label{eq:as_linf_eps_choice_2}
\end{equation}
Inserting \eqref{eq:as_linf_eps_choice_2} into \eqref{eq:as_lp_to_linf_2} yields
\begin{equation}
\frac{d}{dt} \int_{\Omega} e^{\chi_S v} (\as)^{p_k} \, dx  
\leq  -\varepsilon_k \int_{\Omega} e^{\chi_S v} (\as)^{p_k} \, dx + \frac{2(\Cr{c32}\, p_k + \varepsilon_k e^{\chi_S}) \Cr{c34}}{\varepsilon_k} M_{k-1}^2 , \label{eq:as_lp_to_linf_3}
\end{equation}
where
\begin{equation}
M_k = \max \left\{ 1, \sup_{0< t < T}\int_{\Omega} e^{\chi_S v} (\as)^{p_{k}}\right\}.
\end{equation}
Using the same argumentation as in \eqref{eq:ad_lp_to_linf_gronwall}--\eqref{eq:ad_lp_to_linf_final} it follows for $0< t < T$ that also
						\begin{equation}
						\|\as(\cdot,t) \|_\linfo \leq \C, \label{eq:as_lp_to_linf_final}
						\end{equation}
which completes the proof.
\end{proof}

\section{A priori estimate for $\| \nabla v\|_{L^4(\Omega)}$}\label{sec:apriori-v}
We begin by deriving estimates for $\nabla a^D,$ $a^D_t $, $\nabla a^S$ and $a^S_t$.
Let us recall \eqref{eq:mud_bound}
and, hence, by \eqref{eq:a_linfty_bounds} and Lemma \ref{thm:lemma-3.2.Tao} we have
\begin{equation}\label{bounds:sec4}
 \| \ad(t) \|_{L^\infty(\Omega)} ,\ \| \as(t) \|_{L^\infty(\Omega)} \text{ and }\| m(t) \|_{W^1_\infty(\Omega)} \leq C.
\end{equation}

\begin{lem}
 Assume that $(\ad, \as, v, m) \in (\hoels{2}{1})^4$ is a solution of \eqref{eq:system_a}. Then for all $t \in (0,T)$ the following inequalities are fulfilled
 \begin{align}\label{lem:41}
  & \| \nabla \ad(t) \|_{L^2(\Omega)} \leq \C e^{\chi_D \mu_v t},\ \| \ad_t  \|_{L^2(Q_t)} \leq \C t + \C e^{\chi_D \mu_v t}\\
   & \| \nabla \as(t) \|_{L^2(\Omega)} \leq \C e^{\chi_S \mu_v t},\  \| \as_t  \|_{L^2(Q_t)} \leq \C t + \C e^{\chi_S \mu_v t}.
 \end{align}

\end{lem}

\begin{proof}
We begin by multiplying equation for $\ad$ in \eqref{eq:system_a} by $e^{\chi_D v } \ad_t$  and integrating over $\Omega$. We obtain
\begin{multline}\label{p41}\int_\Omega e^{\chi_D v } (\ad_t)^2 \ dx= \int_\Omega \ad_t \nabla \cdot \(e^{\chi_D v} \nabla\ad\) \, dx  +\int_\Omega  e^{\chi_D v } \ad_t \chi_D \ad vm dx\\
- \int_\Omega \memt \, \ad e^{\chi_D v } \ad_t \, dx
					+\int_\Omega e^{\chi_D v } \ad_t (\mu_D-\chi_D \mu_v\,v)\ad \rdev \, dx =: I_1^D + I_2^D + I_3^D + I_4^D.
					\end{multline}
Due to \eqref{eq:system_a}, the bounds from Theorem \ref{loc_existence} and the no-flux boundary condition for $\ad$ we have
\begin{align}\label{I1D:est}
 I_1^D& = \int_\Omega \ad_t \nabla \cdot \(e^{\chi_D v} \nabla\ad\) \, dx \nonumber \\
  &=- \frac{1}{2}\int_\Omega (e^{\chi_D v}\frac{\partial}{\partial t} (| \nabla\ad|^2 ) \, dx\\
  &= - \frac{1}{2}\frac{d}{dt} \int_\Omega e^{\chi_D v}(| \nabla\ad|^2 )\, dx + \frac{\chi_D}{2} \int_\Omega e^{\chi_D v} | \nabla\ad|^2  v_t \, dx \nonumber  \\
   &= - \frac{1}{2}\frac{d}{dt} \int_\Omega e^{\chi_D v}(| \nabla\ad|^2 )\, dx + \frac{\chi_D}{2} \int_\Omega e^{\chi_D v} | \nabla\ad|^2 ( -m v + \mu_v\, v \rdev) \, dx \nonumber \\
    &\leq - \frac{1}{2}\frac{d}{dt} \int_\Omega e^{\chi_D v}(| \nabla\ad|^2 )\, dx + \frac{\chi_D\mu_v}{2} \int_\Omega e^{\chi_D v} | \nabla\ad|^2  \, dx \nonumber .
\end{align}
By Cauchy's inequality, the bounds from Theorem \ref{loc_existence} and \eqref{bounds:sec4}  we have
\begin{align}\label{I2D:est}
 I_2^D &= \int_\Omega  e^{\chi_D v } \ad_t \chi_D \ad vm dx\\
 &= \frac{1}{4} \int_\Omega e^{\chi_D v} (\ad_t)^2 \, dx + \chi_D^2 \int_\Omega  e^{\chi_D v} (\ad)^2m^2 v^2\, dx  \nonumber\\
 &= \frac{1}{4} \int_\Omega e^{\chi_D v} (\ad_t)^2 \, dx + \C.
\end{align}
Analogously we obtain using \eqref{eq:memt_bound}
\begin{equation}\label{I3D:est}
 I_3^D = - \int_\Omega \memt \, \ad e^{\chi_D v } \ad_t \, dx \leq \frac{1}{4} \int_\Omega e^{\chi_D v} (\ad_t)^2 \, dx + \C.
\end{equation}
By Cauchy's inequality, the bounds from Theorem \ref{loc_existence} and \eqref{bounds:sec4}  we have
\begin{align}\label{I4D:est}
 I_4^D &= \int_\Omega e^{\chi_D v } \ad_t (\mu_D-\chi_D \mu_v\,v)\ad \rdev\, dx \\
 &\leq \C \int_\Omega e^{\chi_D v } | \ad_t | \, dx \nonumber\\
 &\leq \frac{1}{4}  \int_\Omega e^{\chi_D v } | \ad_t |^2 \, dx +  \C \nonumber.
\end{align}
Inserting \eqref{I1D:est}- \eqref{I4D:est} into \eqref{p41} we obtain
\begin{equation}\label{p411}
 \frac{1}{4} \int_\Omega e^{\chi_D v } (\ad_t)^2 \ dx + \frac{1}{2}\frac{d}{dt} \int_\Omega e^{\chi_D v}(| \nabla\ad|^2 )\, dx \leq \frac{\chi_D\mu_v}{2} \int_\Omega e^{\chi_D v} | \nabla\ad|^2  \, dx + \Cl{110},
\end{equation}
 which implies
\begin{equation}\label{p412}
 \frac{d}{dt} \int_\Omega e^{\chi_D v}(| \nabla\ad|^2 )\, dx \leq {\chi_D\mu_v}\int_\Omega e^{\chi_D v} | \nabla\ad|^2  \, dx + 2\Cr{110}.
\end{equation}
Applying Gronwall's lemma to \eqref{p412} implies
\begin{equation}\label{p413}
\int_\Omega e^{\chi_D v}(| \nabla\ad|^2 )\, dx \leq \C  e^{\chi_D\mu_v t}.
\end{equation}
Integrating both sides of \eqref{p411} in time and using \eqref{p413} gives 
\begin{equation}
  \int_0^t \int_\Omega e^{\chi_D v } (\ad_t)^2 \ dx\, ds \leq 4 \Cr{110} t + \C e^{\chi_D\mu_v t}.
\end{equation}
This completes the proof of the first line of \eqref{lem:41}.
The proof of the second line is obtained analogously by multiplying the equation for $\as$ in \eqref{eq:system_a} by $e^{\chi_S v } \as_t$  and integrating over $\Omega$.

\end{proof}

The following lemma relates $\| \nabla v(t)\|_{L^p(\Omega)}$ with  $\| \nabla \ad (t)\|_{L^p(\Omega)}$ and  $\| \nabla \as (t)\|_{L^p(\Omega)}.$

\begin{lem}
	Assume that $(\ad, \as, v, m) \in (\hoels{2}{1})^4$ is a solution of \eqref{eq:system_a}. % and $d=2$.
%	\begin{equation} \mu_D \geq \chi_D \mu_v, \quad
%	\mu_S \geq \chi_S \mu_v
%	\end{equation}
	Then the following inequality holds
	\begin{equation}\label{eq:dvp_estim}
	\| \nabla v(t) \|_{L^p(\Omega)}^p \leq \C(T,p) \( \|\nabla \ad\|_{L^p(Q_T)}^p + \| \nabla \as\|_{L^p(Q_T)}^p + 1\) \quad \text{ for any } p >1.
	\end{equation}
\end{lem}
\begin{proof}
	%	Due to Theorem \ref{} and Lemma \ref{} we know that
	%	\begin{equation}
	%		\|m(t) \|_{W^1_\infty(\Omega)} \leq C_{200}.
	%	\end{equation} 
	%	Similar as in \eqref{eq:dvdxiODE} we get by differentiation in \eqref{eq:system_a} that
	%	\begin{equation}
	%		(\nabla v)_t = h_{150} \nabla v -h_{151}, \label{eq:nablavODE}
	%	\end{equation}
	%	where 
	%	\begin{align}
	%		h_{150} &= -m + \mu_v \vfree - \mu_v v(1 + \chi_S e^{\chi_S v} a^S + \chi_D e^{\chi_D v} a^D ) \leq \mu_v \\
	%		h_{151} &= v \nabla m + \mu_v v (e^{\chi_S v} \nabla a^S + e^{\chi_D v} \nabla a^D) 
	%	\end{align} 
	We use the chain rule in \eqref{eq:system_a} to obtain
	\begin{equation}\label{p421}
	\nabla v_t=h_{100}\nabla v-\(v\nabla m + \mu_v v e^{\chi_S v}\nabla a^S + \mu_v v e^{\chi_D v}\nabla a^D\)
	\end{equation}
	with 
	\begin{equation}
	\label{p422}
	h_{100}= -m +\mu_v \rdev-\mu_v v e^{\chi_S v}\chi_S a^S  -\mu_v v e^{\chi_D v}\chi_D a^D. 
	\end{equation} 
	
	Further we use equation \eqref{p421} and multiply it by $p\nabla v |\nabla v|^{p-2}$.
	Employing \eqref{bounds:sec4}, the bounds from Theorem  \ref{loc_existence} and Young's inequality we obtain
	\begin{align}
	(|\nabla v|^p)_t &= h_{100} p |\nabla v|^p  - \Big( p v \nabla v \cdot  \nabla m |\nabla v|^{p-2} + p \mu_v v e^{\chi_S v}\nabla \as\cdot \nabla v |\nabla v|^{p-2} +  p \mu_v v e^{\chi_D v} \nabla \ad\cdot \nabla v |\nabla v|^{p-2} \Big)\nonumber\\
	&\leq  \mu_v p |\nabla v|^p + p \| \nabla m\|_{L^\infty(\Omega)} |\nabla v|^{p-1}  + p \mu_v e^{\chi_D} |\nabla \ad| |\nabla v|^{p-1}  + p \mu_v e^{\chi_S} |\nabla \as| |\nabla v|^{p-1}\nonumber\\
	&\leq \C |\nabla v|^p + \C |\nabla \ad|^p + \C |\nabla \as|^p + \C.\label{eq:451}
	\end{align}
By integration over $\Omega$ we get 
	\begin{equation}
	\frac{d}{dt} \int_{\Omega} |\nabla v|^p \, dx \leq \C \( \int_{\Omega} |\nabla v|^p \, dx + \int_{\Omega} |\nabla \ad|^p \, dx + \int_{\Omega} |\nabla \as|^p \, dx+1\),
	\end{equation}
	which yields also
	\begin{equation} \label{eq:dvp_ODE}
	\frac{d}{dt} \(\int_{\Omega} |\nabla v|^p \, dx + 1\) \leq \C \( \int_{\Omega} |\nabla \ad|^p \, dx + \int_{\Omega} |\nabla \as|^p \, dx+ 1 \) \( \int_{\Omega} |\nabla v|^p \, dx+ 1\).
	\end{equation}
	The estimate \eqref{eq:dvp_estim} follows by the Gronwall Lemma applied to \eqref{eq:dvp_ODE}.
\end{proof}

Our next lemma provides ${L^4(Q_T)}$-bounds for $\nabla \ad$, $\nabla \as$ which only depend on $T$, thereby ruling out finite time blowup of these norms.
\begin{lem}
 Assume that $(\ad, \as, v, m) \in (\hoels{2}{1})^4$ is a solution of \eqref{eq:system_a}% and $d=2$
 . Then the following inequalities are satisfied
 \begin{equation}\label{lem:43}
  \int_0^T \| \Delta \ad(t) \|_{L^2(\Omega)}^2\, dt  \leq \C(T) \quad \text{and} \quad \int_0^T \| \Delta \as(t) \|_{L^2(\Omega)}^2\, dt  \leq \C(T)
  \end{equation}
as well as
\begin{equation}\label{lem:46}
	\int_0^T \| \nabla \ad(t) \|_{L^4(\Omega)}^4\, dt  \leq \C(T) \quad \text{and} \quad \int_0^T \| \nabla \as(t) \|_{L^4(\Omega)}^4\, dt  \leq \C(T)
\end{equation}
\end{lem}

\begin{proof}
 Due to the bounds in Theorem \ref{loc_existence} and \eqref{bounds:sec4} we may rewrite the equations for $\ad$, $\as$ of \eqref{eq:system_a} as
 \begin{align}
  \ad_t &= \Delta \ad + \chi_D \nabla v \cdot \nabla \ad + h_{101}, \label{eq:ad}\\
  \as_t &= \Delta \as + \chi_S \nabla v \cdot \nabla \as + h_{102},
 \end{align}
with 
\begin{align}
 \|h_{101}\|_{L^\infty(\Omega)} = \| \chi_D \ad vm - \memt \, \ad 
					+ (\mu_D-\chi_D \mu_v\,v)\ad \rdev\|_{L^\infty(\Omega)} \leq \C, \label{eq:h101}\\
 \|h_{102}\|_{L^\infty(\Omega)} = \| \chi_S \as vm + \memt \, \ad 
					+ (\mu_S-\chi_S \mu_v\,v)\as \rdev\|_{L^\infty(\Omega)} \leq \C.
\end{align}

From equations \eqref{eq:ad}, \eqref{eq:h101} and the estimate \eqref{lem:41} we get for any $0 \leq t \leq T$
 \begin{align}\label{dad:est}
 \int_0^t \| \Delta \ad(s) \|_{L^2(\Omega)}^2\, ds  \leq \C T^2 + \C e^{2 \chi_D T} + 2 \chi_D^2 \int_0^t \| \nabla v \cdot \nabla \ad\|_{L^2(\Omega)}^2\, ds ,\\
 \label{das:est}
 \int_0^t \| \Delta \as(s) \|_{L^2(\Omega)}^2\, ds  \leq \C T^2 + \C e^{2 \chi_S T} + 2 \chi_S^2 \int_0^t \| \nabla v \cdot \nabla \as\|_{L^2(\Omega)}^2\, ds .
 \end{align}
 
 The last term on the right hand side needs to be estimated further.
 Using Hölder's inequality, equation \eqref{eq:dvp_estim} for $p=4$ and 
 \[ \sqrt{ y + z} \leq \sqrt{y} + \sqrt{z}  \quad \forall y, z \geq 0\]
 we obtain the following  estimate for $I \in \{ S,D\}$ and $J\in \{ S,D\}\setminus\{I\}$ for all $0 \leq t \leq T$
 \begin{align}\label{nav:est}
  \int_0^t \| \nabla v \cdot \nabla \ai\|_{L^2(\Omega)}^2\, ds &\leq \int_0^t \| \nabla v\|_{L^4(\Omega)}^2 \| \nabla \ai\|_{L^4(\Omega)}^2\, ds\\
		  &= \int_0^t \Big(\| \nabla v\|_{L^4(\Omega)}^4 \| \nabla \ai\|_{L^4(\Omega)}^4\Big)^{1/2} \, ds\nonumber\\
		  &\leq\Big(\int_0^t 1 \, ds\Big)^{1/2}  \Big(\int_0^t \| \nabla v\|_{L^4(\Omega)}^4 \| \nabla \ai\|_{L^4(\Omega)}^4 \, ds\Big)^{1/2}\nonumber\\
		  &=\sqrt{t} \Big(\int_0^t \| \nabla v\|_{L^4(\Omega)}^4 \| \nabla \ai\|_{L^4(\Omega)}^4 \, ds\Big)^{1/2}\nonumber\\
		  &\leq\sqrt{t} \Big(\int_0^t  e^{\Cl{113} t} \left[ \Cl{114} + \Cl{115} \int_0^s \| \nabla \ai \|_{L^4(\Omega)}^4 +  \| \nabla \aj \|_{L^4(\Omega)}^4\, d\tau \right] \| \nabla \ai\|_{L^4(\Omega)}^4 \, ds\Big)^{1/2}\nonumber\\	
		  &=\sqrt{t} \Big\{\int_0^t \Big[ \Cr{114}  e^{\Cr{113} T} \| \nabla \ai\|_{L^4(\Omega)}^4  + \frac{\Cr{115}}{2}  e^{\Cr{113} T} \frac{d}{ds} \Big( \int_0^s \| \nabla \ai \|_{L^4(\Omega)}^4 \, d\tau\Big)^2\Big]\, ds\nonumber\\
		   & \quad     + \Cr{115} e^{\Cr{113} T} \int_0^t \| \nabla \ai \|_{L^4(\Omega)}^4 \, ds \int_0^t \| \nabla \aj \|_{L^4(\Omega)}^4 \, ds \Big\}^{1/2}\nonumber \\
		   &\leq \sqrt{t} \Big\{ \Cr{114}  e^{\Cr{113} T} \int_0^t \| \nabla \ai\|_{L^4(\Omega)}^4 \, ds + \Cr{115}  e^{\Cr{113} T} \Big( \int_0^t \| \nabla \ai \|_{L^4(\Omega)}^4 \, ds\Big)^2\nonumber\\
		   & \quad   + \Cr{115}  e^{\Cr{113} T} \Big( \int_0^t \| \nabla \aj \|_{L^4(\Omega)}^4 \, ds\Big)^2   \Big\}^{1/2}\nonumber \\
 		   &\leq \sqrt{t} \sqrt{\Cr{114}}  e^{\frac{\Cr{113}}{2} T} \Big(\int_0^t \| \nabla \ai\|_{L^4(\Omega)}^4 \, ds\Big)^{1/2} + \sqrt{t}  \sqrt{\Cr{115}}  e^{\frac{\Cr{113}}{2} T} \int_0^t \| \nabla \ai \|_{L^4(\Omega)}^4 \, ds\nonumber\\
 		   & \quad   +\sqrt{t}  \sqrt{\Cr{115}}  e^{\frac{\Cr{113}}{2} T}  \int_0^t \| \nabla \aj \|_{L^4(\Omega)}^4 \, ds\nonumber \\
 		   &\leq  \sqrt{t} (\sqrt{\Cr{114}} + \sqrt{\Cr{115}} ) e^{\frac{\Cr{113}}{2} T} \int_0^t \| \nabla \ai \|_{L^4(\Omega)}^4 \, ds
 		   + \sqrt{T} \sqrt{\Cr{114}}  e^{\frac{\Cr{113}}{2} T}\nonumber\\
 		   & \quad   + \sqrt{t} \sqrt{\Cr{115}}  e^{\frac{\Cr{113}}{2} T}  \int_0^t \| \nabla \aj \|_{L^4(\Omega)}^4 \, ds\nonumber.
%  		   &=: \sqrt{t} C_{116} e^{C_{117} T} \int_0^t \| \nabla \ai \|_{L^4(\Omega)}^4 \, ds
%  		   + \sqrt{T} \sqrt{\Cr{114}}  e^{C_{117} T}\nonumber\\
%  		   & \quad   + \sqrt{t} C_{118}  e^{C_{117} T}  \int_0^t \| \nabla \aj \|_{L^4(\Omega)}^4 \, ds\nonumber 
 \end{align}
Since we consider the case of two space dimensions, the Gagliardo-Nirenberg inequality, and the estimate $\|D^2 w \|_{L^2(\Omega)} \leq C \|\Delta w\|_{L^2(\Omega)}$ 
for any $w \in H^2(\Omega)$ with $\frac{\partial w}{\partial \nu} =0$ on $\partial \Omega$ imply the following inequalities for 
any $0 \leq t \leq T:$
\begin{align}\label{nad:est}
 \int_0^t \|\nabla  \ad \|_{L^4(\Omega)}^4 \, ds &\leq \C e^{2 \chi_D \mu_v T} \int_0^t \|\Delta  \ad \|_{L^2(\Omega)}^2 \, ds + \C T  e^{4 \chi_D \mu_v T},\\
 \label{nas:est}
  \int_0^t \|\nabla  \as \|_{L^4(\Omega)}^4 \, ds &\leq \C e^{2 \chi_S \mu_v T} \int_0^t \|\Delta \as \|_{L^2(\Omega)}^2 \, ds + \C T  e^{4 \chi_S \mu_v T}.
\end{align}
Inserting \eqref{nad:est} and \eqref{nas:est} into \eqref{nav:est} we obtain
 \begin{multline}\label{nav:est2}
  \int_0^t \| \nabla v \cdot \nabla \ai\|_{L^2(\Omega)}^2\, ds \\ 
  \leq \sqrt{t} \Cl{123} e^{\Cl{124} T} \Big( \int_0^t \|\Delta \ad \|_{L^2(\Omega)}^2 \, ds + \int_0^t \|\Delta \as \|_{L^2(\Omega)}^2 \, ds\Big)+
  \C \sqrt{T} (1 + T) e^{\C T}.
  \end{multline} 
  By taking the maximum of the constants in the individual estimates of $\ad$ and $\as$, we obtain the same constants in \eqref{nav:est2}.  Inserting \eqref{nav:est2} into \eqref{dad:est} implies
   \begin{align}\label{dad:est1}
 \int_0^t \| \Delta \ad \|_{L^2(\Omega)}^2\, ds  \leq  2 \sqrt{t} \chi_D^2 \Cr{123}e^{\Cr{124} T} \Big( \int_0^t \|\Delta \ad \|_{L^2(\Omega)}^2 \, ds + \int_0^t \|\Delta \as \|_{L^2(\Omega)}^2 \, ds\Big) + \C(T) \\
 \label{dad:est2}
 \int_0^t \| \Delta \as \|_{L^2(\Omega)}^2\, ds  \leq 2 \sqrt{t} \chi_S^2 \Cr{123}e^{\Cr{124} T} \Big( \int_0^t \| \Delta\ad \|_{L^2(\Omega)}^2 \, ds + \int_0^t \| \Delta\as \|_{L^2(\Omega)}^2 \, ds\Big) + \C(T).
 \end{align}
Adding the two estimates above yields
\begin{multline}
 \label{suma:est}
 \int_0^t \| \Delta \ad \|_{L^2(\Omega)}^2\, ds + \int_0^t \| \Delta \as \|_{L^2(\Omega)}^2\, ds\\
 \leq   \sqrt{t}( \chi_D^2 + \chi_S^2)  \Cr{123}e^{\Cr{124} T} \Big( \int_0^t \|\Delta \ad \|_{L^2(\Omega)}^2 \, ds + \int_0^t \|\Delta \as \|_{L^2(\Omega)}^2 \, ds\Big) + \C(T) ,
\end{multline}
so that
\begin{equation}
 \Big( 1 - 2 \sqrt{t} ( \chi_D^2 + \chi_S^2) \Cr{123}e^{\Cr{124} T}\Big)  \Big( \int_0^t \| \Delta\ad \|_{L^2(\Omega)}^2 \, ds + \int_0^t \| \Delta\as \|_{L^2(\Omega)}^2 \, ds\Big) \leq \C(T).
\end{equation}
Choosing 
\[ t_1 =t_1 (T) = \frac{1}{\Big( 4( \chi_D^2 + \chi_S^2) \Cr{123}e^{\Cr{124} T}\Big)^2}\]
we obtain
\begin{equation} \label{suma:est2}
	\int_0^t \| \Delta\ad \|_{L^2(\Omega)}^2 \, ds + \int_0^t \| \Delta \as \|_{L^2(\Omega)}^2 \, ds \leq 2 \C(T) \quad \text{for all } 0 \leq t \leq t_1\,.
\end{equation}
If $t_1(T) \geq T$ we have completed the proof of the lemma. If $t_1(T) < T$ we may repeat the procedure described above by taking $t_0=t_1(T)$ as new initial datum. Since $t_1(T)$ only depends on $T$  we can extend the estimate
 \eqref{suma:est} to the whole time interval $[0,T]$ after finitely many steps. This completes the proof of \eqref{lem:43}. The bounds now \eqref{lem:46} follow by combing \eqref{lem:43} and \eqref{nad:est},\eqref{nas:est}.
\end{proof}

We are now in position to state the main result of this section, i.e., $\| \nabla v(\cdot) \|_{L^4(\Omega)}$ does not blow up in finite time.

\begin{lem}
	Assume that $(\ad, \as, v, m) \in (\hoels{2}{1})^4$ is a solution of \eqref{eq:system_a}. Then the following inequality is fulfilled
	\begin{equation}\label{lem:45}
	\| \nabla v(t) \|_{L^4(\Omega)} \leq \C(T)\, .
	\end{equation}
\end{lem}
\begin{proof}
	Follows directly by combining \eqref{lem:46} with \eqref{eq:dvp_estim}.
\end{proof}

%%%%%%%%%%%%%%%%%%%%%%%%%%%%%%%%%%%%%%%%%%%%%%%%%%%%%%%%%%%%%%%%%%%%%%%%%%%%%%%%%%%%%%%%%%%%%%%%%%%%%%%%%%%%%%%%%%%%%%%%%%%%%%%%%%%%%%%%%%%%%%%%%%%%%%%%%%%%%%%%%%%%%%%%%%%%%%%%%%%%%%%%%%%%%%

%%%%%%%%%%%%%%%%%%%%%%%%%%%%%%%%%%%%%%%%%%%%%%%%%%%%%%%%%%%%%%%%%%%%%%%%%%%%%%%%%%%%%%%%%%%%%%%%%%%%%%%%%%%%%%%%%%%%%%%%%%%%%%%%%%%%%%%%%%%%%%%%%%%%%%%%%%%%%%%%%%%%%%%%%%%%%%%%%%%%%%%%%%%%%
\section{Proof of the global existence Theorem \ref{thm:global.existence}}\label{sec:global}
In this section we show existence and uniqueness of classical solutions of \eqref{eq:system_a} based on the local well-posedness results and a priori estimates from the previous sections.
We begin by establishing uniform in time bounds for $\|\ad(\cdot)\|_{C^2(\Omega)}, \|\as(\cdot)\|_{C^2(\Omega)}, \|v(\cdot)\|_{C^1(\Omega)}, \|m(\cdot)\|_{C^2(\Omega)}$.

\begin{lem}\label{lem:final_a_priori}
Let $(\ad,\as,v,m)\in \(\mathcal C^{2,1}(Q_T)\)^4$ be a solution of  \eqref{eq:system_a}, and
let \eqref{eq:mud_bound} hold.
	Then for all $t \in (0,T)$
		\begin{equation} \label{eq:final_a_priori}
		\|\ad(t)\|_{C^2(\Omega)}, \|\as(t)\|_{C^2(\Omega)}, \|v(t)\|_{C^1(\Omega)}, \|m(t)\|_{C^2(\Omega)} \leq \C(T).
		\end{equation}
\end{lem}
\begin{proof}
Using \eqref{eq:system_a} we can rewrite the equations for $\ad$ and $\as$ as
\begin{align}
	\ad_t &= \Delta \ad + \chi_D \nabla v \cdot \nabla \ad + h_{200}\ad \quad \text{in }\Omega \times (0,T), \label{eq:ad_para_final}\\
	\as_t &= \Delta \ad + \chi_S \nabla v \cdot \nabla \as + h_{201}\as + \memt \, \ad \quad \text{in }\Omega \times (0,T), \label{eq:as_para_final}
\end{align}
where 
\begin{align}
	h_{200} &= -\memt+(\mu_D - \chi_D \mu_v v) \rdev + \chi_D v m, \\
	h_{201} &= (\mu_S - \chi_S \mu_v v) \rdev - \chi_S v m.
\end{align}
By employing \eqref{eq:dvp_estim} for $p=4$, $0\leq v \leq 1,$ \eqref{bounds:sec4}, and \eqref{eq:memt_bound} we have
\begin{equation} \label{eq:para_coeff_lp_final}
	\|\nabla v(t)\|_{L^4(\Omega)},~\|h_{200}(t)\|_\linfo, ~\|h_{201}(t)\|_\linfo, ~ \| \memt\, \ad(t)\|_\linfo \leq \C(T). 
\end{equation}
This allows us to use the maximal parabolic regularity result in $L^p$, see \ref{thm:parabolic_hoel}, for both equations \eqref{eq:ad_para_final}, \eqref{eq:as_para_final} to obtain
\begin{equation}
	\| \ad \|_{W^{2,1}_4(Q_T)}, \| \as \|_{W^{2,1}_4(Q_T)} \leq \Cl{A9}(T).
\end{equation}
Thanks to the Sobolev embedding \ref{thm:sob_embed_gradient} we get for all $p$ a constant $\Cl{152}(p)$ such that
\begin{equation}
	\|\nabla \ad \|_{L^p(Q_T)}, \|\nabla \as \|_{L^p(Q_T)} \leq \Cr{152}(p) \Cr{A9}(T)\quad \text{for all } p>1,
\end{equation}
which yields together with \eqref{eq:dvp_estim} that
\begin{equation} \label{eq:nabvpestim}
	\|\nabla v(t) \|_{L^p(\Omega)} \leq \C(p,T) \quad \text{for all } p>1.
\end{equation}
Using Theorem \ref{thm:parabolic_hoel} again for \eqref{eq:ad_para_final}, \eqref{eq:as_para_final} together with \eqref{eq:para_coeff_lp_final} and \eqref{eq:nabvpestim}, we get
\begin{equation} \label{eq:sobbound_a}
	\| \ad \|_{W^{2,1}_p(Q_T)}, \| \as \|_{W^{2,1}_p(Q_T)} \leq \C(p,T) \quad \text{for all } p>1.
\end{equation}
Moreover, applying \ref{thm:parabolic_hoel} again in equation for $m$ in \eqref{eq:system_a} we obtain
\begin{equation} \label{eq:sobbound_m}
	\| m \|_{W^{2,1}_p(Q_T)} \leq \C(p) \quad \text{for all } p>1.
\end{equation}
Applying the Sobolev embedding \ref{thm:sob_embed_hoelder} to \eqref{eq:sobbound_a}, \eqref{eq:sobbound_m} for a fixed $p> 5$ yields for $\lambda = 1 - 5/p$
\begin{equation} \label{eq:final_hoelder_estim}
	\| \ad \|_\hoeloh{\lambda}, ~ \| \as \|_\hoeloh{\lambda}, ~ \| m \|_\hoeloh{\lambda} \leq \C(T).
\end{equation}
By considering $0\leq v \leq1$, the equation for $\as$ in \eqref{eq:system_a} together with \eqref{eq:final_hoelder_estim} as well as \eqref{eq:dvdxi_sol}, \eqref{eq:help1dvdxi}, and \eqref{eq:help2dvdxi} with \eqref{eq:final_hoelder_estim} we get
\begin{equation} \label{eq:vc11}
	\|v \|_\coo \leq \C(T).
\end{equation}
Using now the same arguments as in the proof of Theorem \ref{thm:reg_local} we obtain
\begin{equation} \label{eq:final_c2_estim}
\| \ad \|_{C^{2+\lambda,1+\lambda/2}(\bar Q_{T})}, ~ \| \as \|_{C^{2+\lambda,1+\lambda/2}(\bar Q_{T})}, ~ \| m \|_{C^{2+\lambda,1+\lambda/2}(\bar Q_{T})} \leq \C(T).
\end{equation}
Estimate \eqref{eq:final_a_priori} follows from \eqref{eq:final_c2_estim} and \eqref{eq:vc11}.
\end{proof}

%Finally, we can prove existence and uniqueness of global classical solutions.
%\begin{theorem}[Global existence]\label{thm:global.existence}
%	Given that $d=2,~ \mu_D \geq \chi_D \mu_v,$ and $\mu_S \geq \chi_S \mu_v,$ 
%	then there exists a unique solution $(\cd, \cs, v, m) \in (\hoels{2}{1}) ^4$ with $\cd, \cs, m \geq 0$ and $0 \leq v \leq 1$ of system \eqref{eq:reduced_II} for any $T>0$.
%\end{theorem}

Finally we can prove the existence and uniqueness of the global classical solutions, as stated in the main Theorem \ref{thm:global.existence}. 

\begin{proof}[Proof of the main Theorem \ref{thm:global.existence}]
	Due to the equivalence of \eqref{eq:reduced_II} and \eqref{eq:system_a} the proof is a consequence of Theorem \ref{loc_existence}, Theorem \ref{thm:reg_local} and Lemma \ref{lem:final_a_priori}.
	Indeed we know that there exist (regular) local-in-time solutions due to Theorem \ref{loc_existence} and Theorem \ref{thm:reg_local}. If they only existed until some maximal final time $T_{max}< \infty$,
	then the a priori bounds in Lemma \ref{lem:final_a_priori} would enable us to use Theorem \ref{loc_existence} in order to extend the solution beyond $T_{max}$ and Theorem  \ref{thm:reg_local} would
	ensure the regularity of this extension.
	This shows that there cannot be a finite maximal time of existence.
\end{proof}

%%%%%%%%%%%%%%%%%%%%%%%%%%%%%%%%%%%%%%%%%%%%%%%%%%%%%%%%%%%%
%%%%%%%%%%%%%%%%%%%%%%%%%%%%%%%%%%%%%%%%%%%%%%%%%%%%%%%%%%%%
% appendix 
\appendix
\section{Parabolic theory}
We consider the problem
\begin{align}
	u_t - D \Delta u + \sum_{i=1}^{d} a_i \frac{\partial u}{\partial_{x_i}} + au &= f \text{ in } Q_T, \label{eq:lin_parabolic}\\
	\partial_\nu u &= 0 \text{ on }\partial \Omega \times (0,T), \\
	u(\cdot,0) &= u_0 \text{ in } \Omega, \label{eq:lin_parabolic_ini}
\end{align}
where $D\in \mathbb{R}^+$, and $a, a_i$ are real valued functions in $Q_T$.
For the initial condition we assume for a fixed $\lambda \in (0,1)$ that
\begin{equation}
	u_0(x) \geq 0, \quad u_0 \in C^{2 + \lambda}(\Omega), \label{eq:lin_parabolic_ini_reg}
\end{equation}
and the compatibility condition
\begin{equation}
	\partial_\nu u_0 = 0. \label{eq:lin_parabolic_ini_comp}
\end{equation}
Furthermore, we assume a bounded domain $\Omega$ with
\begin{equation}
	\partial \Omega \in  C^{2 + \lambda},
\end{equation}
and $Q_T = \Omega \times (0,T)$.
\begin{theorem}\label{thm:parabolic_lp}
	If we assume \eqref{eq:lin_parabolic_ini_reg}, \eqref{eq:lin_parabolic_ini_comp}, and moreover $$a~, a_i,~f \in L^p(Q_T)  \quad 1\leq i \leq d, \quad 0 < T < 1,\quad \partial{\Omega} \in C^2,$$ then the problem \eqref{eq:lin_parabolic}-- \eqref{eq:lin_parabolic_ini} has a unique solution
	$$u \in \sobto,$$ which can be bound by
	$$ \| u\|_\sobto \leq C(\| f \|_{L^p(\Omega)}, ~ \|u_0\|_{L^p(\Omega)}), $$
	\begin{proof}
		Follows from \cite[Theorem 9.1 p. 342]{ladyz}.
	\end{proof}
\end{theorem}

\begin{theorem}\label{thm:parabolic_hoel}
	Assume that, $$a~, a_i,~f \in \hoels{\lambda}{\lambda/2}  \quad 1\leq i \leq d, \quad 0 < T < 1,\quad \partial \Omega \in C^{2+\lambda}$$ and that \eqref{eq:lin_parabolic_ini_reg}, \eqref{eq:lin_parabolic_ini_comp} are satisfied. Then the problem \eqref{eq:lin_parabolic}-- \eqref{eq:lin_parabolic_ini} has a unique solution
	$$u \in \hoels{\lambda + 2 }{\lambda/2 +1}. $$ 
	\begin{proof}
		Follows from {\cite[Theorem 5.3 p. 320]{ladyz}}.
	\end{proof}
\end{theorem}

\begin{theorem}\label{thm:sob_embed_hoelder}
	Assume that $\Omega$ satisfies a weak cone condition and $d\in\{1,2,3\}$. If $p>5$, then
	$$ \|u\|_{C^{1+ \lambda, (1+\lambda)/2}(\bar Q_T)} \leq C \|u \|_{W_p^{2,1}(Q_T)},\quad \lambda = 1- \frac{5}{p},$$
	for all $u \in W_p^{2,1}(Q_T).$
	\begin{proof}
		Follows from \cite[Lemma 3.3 p. 80]{ladyz}. 
	\end{proof}
\end{theorem}

\begin{theorem}\label{thm:sob_embed_gradient}
	Assume that $\Omega$ satisfies a weak cone condition and $d=2$. If $q\geq 4$ then
	$$ \|\nabla u\|_{L^p(\bar Q_T)} \leq C(p) \|u \|_{W_q^{2,1}(Q_T)}, \quad \text{for all } p>4,$$
	for all $u \in W_p^{2,1}(Q_T).$
	\begin{proof}
		Follows from \cite[Lemma 3.3 p. 80]{ladyz}%{ladyzhenskaia1988}. 
	\end{proof}
\end{theorem}
%%%%%%%%%%%%%%%%%%%%%%%%%%%%%%%%%%%%
\section{Proof of Lemma \ref{thm:lemma-3.2.Tao}}\label{sec:proof_of_lemma-3.2.Tao}
\begin{proof}%[Proof of Lemma \ref{thm:lemma-3.2.Tao}]
	%	\env{
	%		The correspondence between the chemical components in the three papers:
	%		\\		
	%		\begin{tabular}{r|l|l}
	%			paper& PDE & representation formula\\
	%			\hline
	%			Kowalcyk: & $c_t =\Delta c + \alpha n - \gamma c$  & $c(t)=e^{-t(A+1)}c_0 + \int_0^t e^{-(t-r)(A+1)} n(r) dr$ \\
	%			&& where $\alpha=\gamma=1$ is assumed\\ 
	%			&&\\
	%			Tao:	& $u_t=\Delta u + c - u$ & $c(t)=e^{-t(A+1)}u_0 + \int_0^t e^{-(t-r)(A+1)} c(r) dr$ \\
	%			&& this is my guess \\ 
	%			&&\\
	%			Our: & $m_t=\Delta m + \cs +\cd - m$ & $c(t)=e^{-t(A+1)}m_0 + \int_0^t e^{-(t-r)(A+1)} (\cs(t) + \cd(t))dr$ \\
	%		\end{tabular}
	%	from Kowalcyk to us :$c\rightarrow m$, $n\rightarrow \cd+\cs$ and $s\rightarrow \rho$.\\
	%	from tao to Kowalcyk : $u\rightarrow c$, $c\rightarrow n$.
	%	}
	
	Let  $A_\rho$ be the sectorial operator defined by $A_\rho u =-\Delta u$ over the domain 
	\[
	D(A_\rho)=\left\{ u\in W^{2,\rho} (\Omega) \text{ with }  \frac{\partial u}{\partial \nu} \Big |_{\Gamma_T}=0 \right\}\,. 
	\]
	We will be needing the following embedding properties of the domains of fractional powers of the operators $A_p+1$:
	\begin{subequations}		
		\begin{align}
			D\((A_p+1)^\beta\)  \hookrightarrow W^1_p(\Omega),&\quad \text{for }\beta>\frac 1 2\, , \label{eq:embedd.1}\\
			D\((A_p+1)^\beta\) \hookrightarrow C^\delta (\Omega),&\quad \text{for }\beta -\frac{d}{2p}>\frac{\delta}{2} \geq 0\,,\label{eq:embedd.2}
		\end{align}
	\end{subequations}
	and refer to \cite{horstmann2005boundedness,henry1981geometric} and the references therein for further details.

	We consider the representation formula for the solution of the equation for $m$ in \eqref{eq:reduced_II}
	$$m(t)=\underbrace{e^{-t(A_\rho+1)}m_0}_{B_1(t)} + \underbrace{\int_0^t  e^{-(t-r)(A_\rho+1)}\(\cd(r)+\cs(r)\)dr}_{B_2(t)}, \qquad t\in(0,T).$$
	%	where
	%	\begin{align*}
	%		B_1(t)&= e^{-t(A_\rho+1)}m_0\,,\\
	%		B_2(t)&= \int_0^t  e^{-(t-r)(A_\rho+1)}\(\cd(r)+\cs(r)\)dr\,
	%	\end{align*}
	To deduce a control over $m$ we consider the two components separately.
	\begin{description}
		\item[] For $B_1(t)$. 
		\begin{itemize}
			\begin{subequations}
				\item If $2\leq q\leq \infty$, then %\comm{it is known (Horstmann Winkler page 5 ii)}
				$B_1$ and $m_0$ have the same regularity, see \cite{horstmann2005boundedness}, and hence
				\begin{equation}\label{eq:21.6a}
					\|B_1(t)\|_{W^1_q(\Omega)}\leq C\, \|m_0\|_{W^1_q(\Omega)}.
				\end{equation}
				\item If $q<2$, then 
				\begin{equation}\label{eq:21.6b}
					\|B_1(t)\|_{W^1_q(\Omega)}\leq\|B_1(t)\|_{W^1_2(\Omega)}\leq C\, \|m_0\|_{W^1_2(\Omega)}.
				\end{equation}
			\end{subequations}
		\end{itemize}	
		
		\item[] For $B_2(t)$.\\
		We consider the analytic semigroup $\(e^{-tA_\rho}\)_{t\geq 0}$, and its properties $\|(A_\rho+1)^\beta e^{-t(A_\rho+1)}u\|_{L^p(\Omega)}  \leq ct^{-\beta} e^{-v_1t} \|u\|_{L^p}$, for all $u\in L^p(\Omega)$, $t\geq 0$, and for some $v_1>0$, and $\|e^{-tA_\rho} u\|_{L^q(\Omega)} \leq ct^{-\frac d 2(\frac 1 p -\frac 1 q)}\|u\|_{L^p(\Omega)}$, for all $t\in(0,1)$ and $1\leq p<q<\infty$, see also \cite{horstmann2005boundedness}.
		
		Accordingly we can write the following $L^\rho$-$L^q$ estimate, for $\tau>0$
		\begin{align*}
			\|(A_\rho+1)^\beta  e^{-2\tau A_\rho}u \|_{L^q(\Omega)}=&\| (A_\rho+1)^\beta  e^{-\tau(A_\rho+1)}e^{-\tau A_\rho} e^{\tau}u \|_{L^q(\Omega)}\\
			\leq & c \tau^{-\beta} e^{-v_1\tau}\|e^{-\tau A_\rho} e^{\tau} u\|_{L^q(\Omega)}\\
			\leq & \tilde c \tau^{-\beta} e^{-v_1\tau} \tau^{-\frac d 2(\frac 1 \rho -\frac 1 q)}\|e^{\tau}u\|_{L^p(\Omega)}\\
			\leq & 	\tilde c \tau^{-\beta -\frac d 2\(\frac 1 \rho -\frac 1 q\)} e^{\(1-v_1\)\tau} \|u\|_{L^p(\Omega)},
		\end{align*}
		or by setting $t=2\tau$,
		\begin{align}
			\|(A_\rho+1)^\beta  e^{-t A_\rho}u \|_{L^q(\Omega)} \leq &\tilde c \(\frac{t}{2}\)^{-\beta -\frac d 2\(\frac 1 \rho -\frac 1 q\)} e^{\(1-v_1\)\frac t 2 } \|u\|_{L^p(\Omega)} \nonumber \\
			\leq & C t^{-\beta -\frac d 2\(\frac 1 \rho -\frac 1 q\)} e^{\(1-\mu \)t  }e^{-\frac t 2} \|u\|_{L^p(\Omega)}\nonumber \\
			\leq & C t^{-\beta -\frac d 2\(\frac 1 \rho -\frac 1 q\)} e^{\(1-\mu \)t  }\|u\|_{L^p(\Omega)} \label{eq:B2_estimate}
		\end{align}
		for some $\mu >0.$
		
		Applying now \eqref{eq:B2_estimate} to $B_2$, it reads
		\begin{align*}
			\|\(A_\rho+1\)^\beta B_2\|_{L^q(\Omega)}&\leq C\int_0^t (t-r)^{-\beta -\frac{d}{2}\(\frac{1}{\rho}-\frac{1}{q}\)}e^{-\mu(t-r)}\|u (r)\|_{L^\rho(\Omega)}dr\\
			&\leq C \sup_t\|u(t)\|_{L^\rho(\Omega)} \int_0^t (t-r)^{-\beta -\frac{d}{2}\(\frac{1}{\rho} -\frac{1}{q}\) }e^{-\mu (t-r)}dr\,,
		\end{align*}
		where the integral is finite, and in effect $B_2(t)\in D((A_\rho+1)^\beta) $, as long as 
		\begin{equation}\label{eq:21.7}
			-\beta -\frac{d}{2}\(\frac 1 \rho - \frac 1 q\)>-1.
		\end{equation}				
		To this end we distinguish the following sub-cases:
		\begin{itemize}
			\item If $\rho < d$ then there exist $\frac 1 2 < \beta < 1$ such that \eqref{eq:21.7} reads
			$$q < \frac{1}{\frac 1 \rho - \frac 1 d + \frac 2 d\(\beta  - \frac 1 2\)}.$$
			By the embedding now \eqref{eq:embedd.1} of the domain of the operator $\(A_q + 1\)^\beta $ we deduce that
			\begin{equation}
				\|B_2(t)\|_{W^1_q(\Omega)} \leq C,
			\end{equation}
			which along with the bounds \eqref{eq:21.6a} and\eqref{eq:21.6b} of $\|B_1\|$ leads to \eqref{eq:21.4}.
			\item If $\rho=d$, the condition \eqref{eq:21.7} recasts into
			$$ \beta < \frac 1 2 + \frac{d}{2q},$$
			which is satisfied by some $\frac 1 2 < \beta < 1$ for every $q> \rho=d$, and thus \eqref{eq:21.4} follows for $q< \infty$.
			\item If $\rho >d$ there by \eqref{eq:21.7} $\beta < 1-\frac{d}{2\rho} +\frac{d}{2q}$ and since $\frac 1 2 + \frac{d}{2q}< 1-\frac{d}{2\rho} +\frac{d}{2q}$ there exist $\beta$ such that
			$$ \frac 1 2 + \frac{d}{2q} < \beta < 1-\frac{d}{2\rho} +\frac{d}{2q},$$
			such that the embedding \eqref{eq:embedd.2} is valid for $\delta=1$, and reads
			$$D\( (A_q +1)^\beta \)\hookrightarrow C^1 (\bar \Omega),$$
			from which \eqref{eq:21.4} yields for $q=\infty$.
		\end{itemize}
		
	\end{description}
\end{proof}
%%%%%%%%%%%%%%%%%%%%%%%%%%%%%%%%%%%%%%%%%%%%%%%%%%%%%%%%%%%%
%%%%%%%%%%%%%%%%%%%%%%%%%%%%%%%%%%%%%%%%%%%%%%%%%%%%%%%%%%%%
% bibliography
\begin{multicols}{2}
	\bibliographystyle{plain}
	\small
	\bibliography{Analysis}

\begin{thebibliography}{10}

\bibitem{Alikakos.1979}
N.D. Alikakos.
\newblock An application of the invariance principle to reaction-diffusion
  equations.
\newblock {\em J. Differ. Equations}, 1979.

\bibitem{Anderson.2000}
A.R.A. Anderson, M.A.J. Chaplain, E.L. Newman, R.J.C. Steele, and A.M.
  Thompson.
\newblock Mathematical modelling of tumour invasion and metastasis.
\newblock {\em Comput. Math. Method.}, 2000.

\bibitem{Bellomo.2008}
N.~Bellomo, N.K. Li, and P.K. Maini.
\newblock On the foundations of cancer modelling: Selected topics,
  speculations, and perspectives.
\newblock {\em Math. Models Methods Appl. Sci.}, 2008.

\bibitem{Blanchet.2006}
A.~Blanchet, J.~Dolbeault, and B.~Perthame.
\newblock {Two-dimensional Keller-Segel model: optimal critical mass and
  qualitative properties of the solutions}.
\newblock {\em Electron. J. Diff. Eqns.}, 2006.

\bibitem{Brabletz.2005}
T.~Brabletz, A.~Jung, S.~Spaderna, F.~Hlubek, and T.~Kirchner.
\newblock {Opinion: migrating cancer stem cells - an integrated concept of
  malignant tumour progression}.
\newblock {\em Nat. Rev. Cancer}, 2005.

\bibitem{Chaplain.2005}
M.A.J. Chaplain and G.~Lolas.
\newblock Mathematical modelling of cancer cell invasion of tissue. the role of
  the urikinase plasminogen activation system.
\newblock {\em Math. Mod. Meth. Appl. S.}, 2005.

\bibitem{Corrias.2004}
L.~Corrias, B.~Perthame, and H.~Zaag.
\newblock Global solutions of some chemotaxis and angiogenesis systems in high
  space dimensions.
\newblock {\em Milan J. Math.}, 2004.

\bibitem{Schmeiser.2009}
J.~Dolbeault and Chr. Schmeiser.
\newblock The two-dimensional keller-segel model after blow-up.
\newblock {\em Discrete Cont. Dyn.-B.}, 2009.

\bibitem{Egeblad.2002}
M.~Egeblad and J.~Werb.
\newblock New functions for the matrix metalloproteinases in cancer
  progression.
\newblock {\em Nat. Rev. Cancer}, 2002.

\bibitem{Gross.1962}
J.~Gross and C.~Lapiere.
\newblock Collagenolytic activity in amphibian tissues: a tissue culture assay.
\newblock {\em Proc Natl Acad Sci USA}, 1962.

\bibitem{Gupta.2009}
P.B. Gupta, C.L. Chaffer, and R.A. Weinberg.
\newblock Cancer stem cells: mirage or reality?
\newblock {\em Nat. Med.}, 2009.

\bibitem{Hanahan.2000}
D.~Hanahan and R.A. Weinberg.
\newblock The hallmarks of cancer.
\newblock {\em Cell}, 2000.

\bibitem{Hellmann.2016}
Nadja Hellmann, Niklas Kolbe, and Nikolaos Sfakianakis.
\newblock A mathematical insight in the epithelial-mesenchymal-like transition
  in cancer cells and its effect in the invasion of the extracellular matrix.
\newblock {\em Bull. Braz. Math. Soc., New Series}, 2016.

\bibitem{henry1981geometric}
D.~Henry.
\newblock Geometric theory of semilinear parabolic systems.
\newblock {\em Lecture notes in mathematics}, 840, 1981.

\bibitem{Hillen.2013}
T.~Hillen, K.J. Painter, and M.~Winkler.
\newblock {Convergence of a cancer invasion model to a logistic chemotaxis
  model}.
\newblock {\em Math. Mod. Meth. Appl. S.}, 2013.

\bibitem{horstmann2005boundedness}
D.~Horstmann and M.~Winkler.
\newblock Boundedness vs. blow-up in a chemotaxis system.
\newblock {\em J. Differ. Equations}, 215(1):52--107, 2005.

\bibitem{Johnston.2010}
M.D. Johnston, P.K. Maini, S~Jonathan-Chapman, C.M. Edwards, and W.F. Bodmer.
\newblock On the proportion of cancer stem cells in a tumour.
\newblock {\em J. Theor. Biol.}, 2010.

\bibitem{KS.1970}
E.F. Keller and L.A. Segel.
\newblock Initiation of slime mold aggregation viewed as an instability.
\newblock {\em J. Theor. Biol.}, 1970.

\bibitem{kowalczyk2008global}
R.~Kowalczyk and Z.~Szyma{\'n}ska.
\newblock On the global existence of solutions to an aggregation model.
\newblock {\em J. Math. Anal. Appl.}, 343(1):379--398, 2008.

\bibitem{ladyz}
O.A Ladyzhenskaia, V.A. Solonnikov, and N.N. Ural'tseva.
\newblock {\em Linear and quasi-linear equations of parabolic type}.
\newblock American Mathematical Soc.

\bibitem{Mani.2008}
S.A. Mani, W.~Guo, M.J. Liao, and et~al.
\newblock {The epithelial-mesenchymal transition generates cells with
  properties of stem cells}.
\newblock {\em Cell}, 2008.

\bibitem{Czochra.2010}
A.~Marciniak-Czochra and M.~Ptashnyk.
\newblock Boundedness of solutions of a haptotaxis model.
\newblock {\em Math. Mod. Meth. Appl. S.}, 2010.

\bibitem{Michor.2008}
F.~Michor.
\newblock Mathematical models of cancer stem cells.
\newblock {\em J. Clin. Oncol.}, 2008.

\bibitem{Patlak.1953}
C.S. Patlak.
\newblock {Random walk with persistence and external bias}.
\newblock {\em Bull. Math. Biophys.}, 1953.

\bibitem{Perthame.2014}
B.~Perthame, F.~Quiros, and J.L. Vazquez.
\newblock {The Hele–Shaw Asymptotics for Mechanical Models of Tumor Growth}.
\newblock {\em Arch. Rational Mech. Anal.}, 2014.

\bibitem{Perumpanani.1996}
A.J. Perumpanani, J.A. Sherratt, J.~Norbury, and H.M. Byrne.
\newblock Biological inferences from a mathematical model for malignant
  invasion.
\newblock {\em Invas. Metast.}, 1996.

\bibitem{Preziosi.2003}
L.~Preziosi.
\newblock {\em Cancer modelling and simulation}.
\newblock CRC Press, 2003.

\bibitem{Rao.2003}
J.S. Rao.
\newblock {Molecular mechanisms of glioma invasiveness: the role of proteases}.
\newblock {\em Nat. Rev. Cancer}, 2003.

\bibitem{Reya.2001}
T.~Reya, S.J. Morrison, M.F. Clarke, and I.L. Weissman.
\newblock {Stem cells, cancer, and cancer stem cells}.
\newblock {\em Nature}, 2001.

\bibitem{Roose.2007}
T.~Roose, S.J. Chapman, and P.K. Maini.
\newblock Mathematical models of avascular tumor growth.
\newblock {\em SIAM Rev.}, 2007.

\bibitem{Condeelis.2011}
E.T. Roussos, J.S Condeelis, and A.~Patsialou.
\newblock {Chemotaxis in cancer}.
\newblock {\em Nat. Rev. Cancer}, 2011.

\bibitem{Sfakianakis.2017}
N.~Sfakianakis, N.~Kolbe, N.~Hellmann, and M.~Luk{\'a}cov{\'a}-Medvid'ov{\'a}.
\newblock A multiscale approach to the migration of cancer stem cells:
  Mathematical modelling and simulations.
\newblock {\em Bull. Math. Biol.}, 2017.

\bibitem{Stevens.2000}
A.~Stevens.
\newblock The derivation of chemotaxis equations as limit dynamics of
  moderately interacting stochastic many-particle systems.
\newblock {\em SIAM J. Appl. Math.}, 2000.

\bibitem{Stinner.2015}
C.~Stinner, C.~Surulescu, and A.~Uatay.
\newblock Global existence for a go-or-grow multiscale model for tumor invasion
  with therapy.
\newblock {\em Math. Mod. Meth. Appl. S.}, 2016.

\bibitem{Szymanska.2009}
Z.~Szymanska, C.M. Rodrigo, M.~Lachowicz, and M.A.J. Chaplain.
\newblock Mathematical modelling of cancer invasion of tissue: the role and
  effect of nonlocal interactions.
\newblock {\em Math Mod.Methods Appl. Sci.}, 2009.

\bibitem{Tao.2009}
Y.~Tao.
\newblock {Global existence of classical solutions to a combined
  chemotaxis–haptotaxis model with logistic source}.
\newblock {\em J. Math. Anal. Appl.}, 2009.

\bibitem{tao2011global}
Y.~Tao.
\newblock Global existence for a haptotaxis model of cancer invasion with
  tissue remodeling.
\newblock {\em Nonlinear Anal.-Real}, 12(1):418--435, 2011.

\bibitem{Vainstein.2012}
V.~Vainstein, O.U. Kirnasovsky, and Y.K. Zvia~Agur.
\newblock Strategies for cancer stem cell elimination: Insights from
  mathematical modelling.
\newblock {\em J. Theor. Biol.}, 2012.

\bibitem{Walker.2007}
C.~Walker and G.F. Webb.
\newblock Global existence of classical solutions for a haptotaxis model.
\newblock {\em SIAM J. Math. Anal.}, 2007.

\bibitem{Winkler.2014}
M.~Winkler and Y.~Tao.
\newblock {Energy-type estimates and global solvability in a two-dimensional
  chemotaxis–haptotaxis model with remodeling of non-diffusible attractant.}
\newblock {\em J. Diff. Eq.}, 2014.

\end{thebibliography}
\end{multicols}

%%%%%%%%%%%%%%%%%%%%%%%%%%%%%%%%%%%%%%%%%%%%%%%%%%%%%%%%%%%%
%%%%%%%%%%%%%%%%%%%%%%%%%%%%%%%%%%%%%%%%%%%%%%%%%%%%%%%%%%%%
%%%%%%%%%%%%%%%%%%%%%%%%%%%%%%%%%%%%%%%%%%%%%%%%%%%%%%%%%%%%
\end{document}